\documentclass[12pt]{article}
\usepackage{amsmath}
\usepackage{amssymb}
\usepackage{amsmath,amssymb,amsbsy,amsfonts,amsthm,latexsym,
amsopn,amstext,amsxtra,euscript,amscd}

\begin{document}

\def\A{\mathbb{A}}
\def\B{\mathbf{B}}
\def \C{\mathbb{C}}
\def \F{\mathbb{F}}
\def \K{\mathbb{K}}

\def \Z{\mathbb{Z}}
\def \P{\mathbb{P}}
\def \R{\mathbb{R}}
\def \Q{\mathbb{Q}}
\def \N{\mathbb{N}}
\def \Z{\mathbb{Z}}

\def\B{\mathcal B}
\def\e{\varepsilon}

\def\cA{{\mathcal A}}
\def\cB{{\mathcal B}}
\def\cC{{\mathcal C}}
\def\cD{{\mathcal D}}
\def\cE{{\mathcal E}}
\def\cF{{\mathcal F}}
\def\cG{{\mathcal G}}
\def\cH{{\mathcal H}}
\def\cI{{\mathcal I}}
\def\cJ{{\mathcal J}}
\def\cK{{\mathcal K}}
\def\cL{{\mathcal L}}
\def\cM{{\mathcal M}}
\def\cN{{\mathcal N}}
\def\cO{{\mathcal O}}
\def\cP{{\mathcal P}}
\def\cQ{{\mathcal Q}}
\def\cR{{\mathcal R}}
\def\cS{{\mathcal S}}
\def\cT{{\mathcal T}}
\def\cU{{\mathcal U}}
\def\cV{{\mathcal V}}
\def\cW{{\mathcal W}}
\def\cX{{\mathcal X}}
\def\cY{{\mathcal Y}}
\def\cZ{{\mathcal Z}}

\def\f{\frac{|\A||B|}{|G|}}
\def\AB{|\A\cap B|}
\def \Fq{\F_q}
\def \Fqn{\F_{q^n}}

\def\({\left(}
\def\){\right)}
\def\fl#1{\left\lfloor#1\right\rfloor}
\def\rf#1{\left\lceil#1\right\rceil}
\def\Res{{\mathrm{Res}}}

\newcommand{\comm}[1]{\marginpar{
\vskip-\baselineskip 
\raggedright\footnotesize
\itshape\hrule\smallskip#1\par\smallskip\hrule}}

\newtheorem{lem}{Lemma}
\newtheorem{lemma}[lem]{Lemma}
\newtheorem{prop}{Proposition}
\newtheorem{proposition}[prop]{Proposition }
\newtheorem{thm}{Theorem}
\newtheorem{theorem}[thm]{Theorem}
\newtheorem{cor}{Corollary}
\newtheorem{corollary}[cor]{Corollary}
\newtheorem{prob}{Problem}
\newtheorem{problem}[prob]{Problem}
\newtheorem{ques}{Question}
\newtheorem{question}[ques]{Question}
\newtheorem{rem}{Remark}

\title{Sumsets of reciprocals in prime fields and multilinear Kloosterman sums}
\author{{J.~Bourgain}\\
\normalsize{Institute for Advanced Study,}\\
\normalsize {Princeton, NJ 08540, USA}\\
\normalsize{\tt bourgain@ias.edu} \\
\and\\
{M. Z. Garaev}
\\
\normalsize{Centro de Ciencias Matem\'{a}ticas,}
\\
\normalsize{Universidad Nacional Aut\'onoma de M\'{e}xico,}\\
\normalsize{Morelia 58089, Michoac\'{a}n, M\'{e}xico}\\
\normalsize{\tt garaev@matmor.unam.mx}
 }

%\small Mathematics Subject Classifications: 11L05}

\date{\empty}

\pagenumbering{arabic}

\maketitle

\begin{abstract}
We obtain new results on additive properties of the set
$$
I^{-1}= \{x^{-1}: \quad x\in I\}
$$
where $I$ is an arbitrary interval in the field of residue classes modulo a large prime $p$. We combine our results
with multilinear exponential sum estimates and obtain new results on incomplete multilinear Kloosterman sums.
\end{abstract}
\maketitle

\newpage

\section{Introduction}

In what follows, $\varepsilon>0$ is an arbitrary fixed constant, $\F_p$ is the field of residue classes modulo a large prime $p$ which frequently will be associated with the set $\{0,1,\ldots, p-1\}$. Given an integer $x$ coprime to $p$ (or an element $x$ from $\F_p^{*}=\F_p\setminus\{0\}$) we use $x^*$ or $x^{-1}$ to denote its multiplicative inverse modulo~$p$.

Let $I$ be a non-zero interval in $\F_p$. Additive properties of the reciprocal-set
$$
I^{-1}=\{x^{-1}:\quad x\in I\},
$$
with a subsequent application to Kloosterman sums have been considered in~\cite{B1}.
Among other results, it has been shown there
that for any $\delta>0$ there exists $k\in \Z_{+}$ such that the sumset
$$
k(I^{-1})=\{x_1^{-1}+\ldots+x_{k}^{-1}:\quad x_i\in I\}
$$
satisfies
$$
|k(I^{-1})|>p^{-\delta}\min\{|I|^2, p\}.
$$
In the most interesting case  $|I|<p^{1/2}$ this implies that $|k(I^{-1})|>|I|^2 p^{-\delta}$.
From some recent results in~\cite{CillGar} (see Lemma~\ref{lem:CillGar} below) it follows that
\begin{equation}
\label{eqn:CillGar I*+I*}
|I^{-1}+I^{-1}|>\min\{|I|^2, \sqrt{p |I|}\}|I|^{o(1)}.
\end{equation}
In particular, if $|I|<p^{1/3}$, then
$$
|I^{-1}+I^{-1}|>|I|^{2+o(1)}.
$$

The aim of the present paper is to establish new additive properties of the set $I^{-1}$. We then combine our results with
recent estimates of multilinear exponential sum bounds from~\cite{B2}  and obtain new results on multilinear Kloosterman sums.

The structure of the paper is as follows. In section 2 we state our results on additive properties of the set $I^{-1}$
and on estimates of Kloosterman sums. In section 3 we give some basic preliminaries which are used throughout the paper. In sections 4--8
we give some backgrounds and prove preliminary lemmas. The proof of our results on additive properties of reciprocals on intervals (Theorems~\ref{thm:kI*}-\ref{thm:kI*=kI* I=[1,N] prime}) is given in section 9. The proof of Theorems~\ref{thm:Kloost double 18/37}-\ref{thm:Kloost 1/2} are given in section 10.
In section~11 we give the proof of Theorem~\ref{thm:Archimed} on Archimedian counterpart of Karatsuba's estimate, in section~12 we give the proof of Theorem~\ref{thm:pix-pix-y} on $\pi(x)-\pi(x-y)$, Theorem~\ref{thm:linear sqrt log p} on a linear Kloosterman sums and Theorem~\ref{thm:BrunTitch} on Brun-Titchmarsh theorem.

In this paper we consider only the case of prime modulus. The case of composite modulus will be considered in a forthcoming paper.

\section{Statement of results}

\subsection{Reciprocals of intervals}

We recall that $I$ denotes an arbitrary non-zero interval in $\F_p$. We first start with results on additive properties of $I^{-1}$.

\begin{theorem}
\label{thm:kI*} For any fixed positive integer constant $k$ the number $J_{2k}$ of solutions of the congruence
$$
x_1^{-1}+\ldots+x_k^{-1}=x_{k+1}^{-1}+\ldots+x_{2k}^{-1},\qquad x_1,\ldots,x_{2k}\in I,
$$
satisfies
\begin{equation}
\label{eqn:thm kI*}
J_{2k}<\Bigl(|I|^{2k^2/(k+1)}+\frac{|I|^{2k}}{p}\Bigr)|I|^{o(1)}.
\end{equation}
\end{theorem}

Recall that~\eqref{eqn:thm kI*} is equivalent to saying that for any $\varepsilon>0$ there exists $c=c(k;\varepsilon)>0$ such that
$$
J_{2k}<c\Bigl(|I|^{2k^2/(k+1)}+\frac{|I|^{2k}}{p}\Bigr)|I|^{\varepsilon}.
$$

\begin{corollary}
\label{cor:thm kI*}
Let $|I|<p^{1/2}$. Then for any fixed positive integer constant $k$,
$$
|k(I^{-1})|>|I|^{2k/(k+1)+o(1)}.
$$
\end{corollary}

We remark that for $k=3$ one can prove the bound
$$
|I^{-1}+I^{-1}+I^{-1}|>|I|^{1.55+o(1)}.
$$

We next consider a ternary additive congruence
with $I^{-1}$.

\begin{theorem}
\label{thm:3I* small |I|} Let $|I|<p^{3/46}$. Then for any element $\lambda\in \F_p$ with
\begin{equation}
\label{eqn:lambda not int I* and 0}
\lambda\not\in I^{-1}\cup\{0\},
\end{equation}
the number $J$ of solutions of the congruence
$$
x^{-1}+y^{-1}+z^{-1}=\lambda,\quad x,y,z\in I,
$$
satisfies
$$
J<|I|^{2/3+o(1)}.
$$
\end{theorem}

The restriction~\eqref{eqn:lambda not int I* and 0} is motivated by the possibility of $|I|^{1+o(1)}$ solutions otherwise (for instance,  $z=\lambda^{-1}$ and $x+y=0$).

From Theorem~\ref{thm:3I* small |I|} it easily follows that for  $|I|<p^{3/46}$, one has the bound
$$
|I^{-1}+I^{-1}+I^{-1}|>|I|^{7/3+o(1)}.
$$
The following statements show that for sufficiently small $I$ one has optimal bounds.

\begin{theorem}
\label{thm:3I*=3I* small |I|}
Let $|I|<p^{1/18}$.
Then the number $J_6$ of solutions of the congruence
$$
x_1^{-1}+x_2^{-1}+x_3^{-1}= x_4^{-1}+x_5^{-1}+x_6^{-1},\qquad x_1,\ldots, x_6\in I,
$$
satisfies
$$
J_6<|I|^{3+o(1)}.
$$
In particular, for $|I|<p^{1/18}$ we have
$$
|I^{-1}+I^{-1}+I^{-1}|>|I|^{3+o(1)}.
$$
\end{theorem}

\begin{theorem}
\label{thm:kI*=kI* small |I|}
There is an absolute constant $c>0$ such that for any fixed positive integer constant $k$ and any interval $I\subset \F_p$ with $|I|<p^{c/k^2}$
the number $J_{2k}$ of solutions of the congruence
$$
x_1^{-1}+\ldots+x_k^{-1}= x_{k+1}^{-1}+\ldots+x_{2k}^{-1},\qquad x_1,\ldots, x_{2k}\in I,
$$
satisfies
$$
J_{2k}<|I|^{k+o(1)}.
$$
In particular, for such intervals $I$ we have
$$
|k(I^{-1})|>|I|^{k+o(1)}.
$$
\end{theorem}

\begin{rem}
From the proof it is clear that in Theorem~\ref{thm:kI*=kI* small |I|}  one can take $c=1/4$.
\end{rem}

\begin{theorem}
\label{thm:kI*=kI* I=[1,N]} Let $I=[1,N]$. Then
the number $J_{2k}$ of solutions of the congruence
$$
x_1^*+\ldots+x_k^*\equiv x_{k+1}^*+\ldots+x_{2k}^*\pmod p,\qquad x_1,\ldots, x_{2k}\in I,
$$
satisfies
$$
J_{2k}<(2k)^{90k^3}(\log N)^{4k^2}\Bigl(\frac{N^{2k-1}}{p}+1\Bigr)N^{k}.
$$
\end{theorem}

We also give a version of Theorem~\ref{thm:kI*=kI* I=[1,N]}, where the variables $x_j$ are restricted to prime numbers.
By $\cP$ we denote the set of primes.

\begin{theorem}
\label{thm:kI*=kI* I=[1,N] prime} Let $I=[1,N]$. Then
the number $J_{2k}$ of solutions of the congruence
$$
x_1^*+\ldots+x_k^*\equiv x_{k+1}^*+\ldots+x_{2k}^*\pmod p,\qquad x_1,\ldots, x_{2k}\in I\cap\cP,
$$
satisfies
$$
J_{2k}<(2k)^{k}\Bigl(\frac{N^{2k-1}}{p}+1\Bigr)N^{k}.
$$
\end{theorem}

\subsection{Incomplete multilinear Kloosterman sums}

Below we use the abbreviation $e_p(z)=e^{2\pi i z/p}$. The incomplete Kloosterman sums
$$
\sum_{x=M+1}^{M+N}e_p(ax^*+bx),
$$
where $a$ and $b$ are integers, $\gcd(a,p)=1$, are well known in the literature, with a variety of applications. These sums
 are estimated by $O(p^{1/2}\log p)$ as a consequence of Weil bounds. For $M=0$ and $N$ very small (that is, $N=p^{o(1)}$) these sums have been estimated by Korolev~\cite{Kor}.

The incomplete bilinear Kloosterman sums
$$
S=\sum_{x_1=M_1+1}^{M_1+N_1}\,\sum_{x_2=M_2+1}^{M_2+N_2}\alpha_1(x_1)\alpha_2(x_2)e_p(ax_1^*x_2^*),
$$
where $\alpha_i(x_i)\in\C,\, |\alpha_i(x_i)|\le 1$, are also well known in the literature. Observe, that if one of the parameters $N_i$ is much larger than $p^{1/2}$, then $S$ can easily be estimated. For instance, if, say, $N_1^{1-c}>p^{1/2}$ for some $c>0$, one can use the Weil bound and get
\begin{equation*}
\begin{split}
|S|^2\le &N_1\sum_{x_1=M_1+1}^{M_1+N_1}\Bigl|\sum_{x_2=M_2+1}^{M_2+N_2}\alpha_2(x_2)e_p(ax_1^*x_2^*)\Bigr|^2\le\\
&N_1\sum_{y=M_2+1}^{M_2+N_2}\,\sum_{z=M_2+1}^{M_2+N_2}\Bigl|\sum_{x_1=M_1+1}^{M_1+N_1}e_p(ax_1^*(y^*-z^*)\Bigr|\ll
N_1^2N_2+N_1N_2^2\sqrt{p}(\log p),
\end{split}
\end{equation*}
which implies that
$$
|S|<\Bigl(N_2^{-1/2}+N_1^{-c/2}\Bigr)(N_1N_2)^{1+o(1)}.
$$
Thus, the most nontrivial case is $N_i<p^{1/2}$. When $M_1=M_2=0$ the sum $S$ (in a more general form in fact) has been estimated by Karatsuba~\cite{Kar1, Kar2} for very short ranges of $N_1$ and $N_2$, and by Bourgain~\cite{B1} for arbitrary $M_1,M_2$ provided that $N_1N_2>p^{1/2+\varepsilon}$. A full explicit version of Bourgain's result has been given by Baker~\cite{Bak}.

The incomplete $n$-linear Kloosterman sums
$$
\sum_{x_1=M_1+1}^{M_1+N_1}\ldots\sum_{x_n=M_n+1}^{M_n+N_n}e_p(a_1x_1+\ldots+a_nx_n+a_{n+1}(x_1\ldots x_n)^{*}),
$$
where $a_i\in\Z,\, \gcd(a_{n+1},p)=1,$ have been studied  by Luo~\cite{Luo} and Shparlinski~\cite{Shp1} for arbitrary $n$.
The main tool they used are the bounds of Burgess~\cite{Burgess} on incomplete Gauss sums.

Here, we combine our Theorems~\ref{thm:kI*},~\ref{thm:3I*=3I* small |I|},~\ref{thm:kI*=kI* small |I|},~\ref{thm:kI*=kI* I=[1,N]} with the multilinear exponential sum bounds from~\cite{B2} (see Lemma~\ref{lem:B2} below) and obtain new estimates on Kloosterman sums. In what follows, $\alpha_1(x_1),\ldots,\alpha_n(x_n)$ are arbitrary complex numbers with $|\alpha_i(x_i)|\le 1$.

\begin{theorem}
\label{thm:Kloost double 18/37} For any intervals $I_1, I_2$ with
$$
|I_1|>p^{1/18},\qquad |I_2|>p^{5/12+\varepsilon}
$$
we have
$$
\max_{(a,p)=1}\Bigr|\sum_{x_1\in I_1}\sum_{x_2\in I_2}\alpha_1(x_1) \alpha_2(x_2)e_p(ax_1^*x_2^*)\Bigl|<p^{-\delta}|I_1||I_2|
$$
for some $\delta=\delta(\varepsilon)>0$.
\end{theorem}

Note that $|I_1||I_2|=p^{1/2-1/36+\varepsilon}$.

\begin{rem}
The statement of Theorem~\ref{thm:Kloost double 18/37}
remains true for the general sum
$$
\sum_{x_1=1}^{N_1}\sum_{x_2=1}^{N_2}\alpha_1(x_1) \alpha_2(x_2)e_p(ax_1^*x_2^*+bx_1x_2).
$$
This can be achieved incorporating~\cite[Lemma A.8]{B1}.
\end{rem}

When $M_1=M_2=0$, we prove the following result, expanding the range of applicability of Karatsuba's
estimate~\cite{Kar1}.

\begin{theorem}
\label{thm:Kloost Karatsuba range}
Let $I_1=[1, N_1],\, I_2=[1,N_2]$. Then uniformly over all positive integers $k_1,k_2$ and $\gcd(a,p)=1$ we have
\begin{equation*}
\begin{split}
\Bigl|\sum_{x_1\in I_1} \sum_{x_2\in I_2}\alpha_1(x_1) &\alpha_2(x_2)e_p(ax_1^*x_2^*)\Bigr|
< (2k_1)^{\frac{45k_1^2}{k_2}}(2k_2)^{\frac{45k_2^2}{k_1}}(\log p)^{2(\frac{k_1}{k_2}+\frac{k_2}{k_1})}\times \\
\times &\Bigl(\frac{N_1^{k_1-1}}{p^{1/2}}+\frac{p^{1/2}}{N^{k_1}}\Bigr)^{1/(2k_1k_2)}
\Bigl(\frac{N_2^{k_2-1}}{p^{1/2}}+\frac{p^{1/2}}{N^{k_2}}\Bigr)^{1/(2k_1k_2)}N_1N_2.
\end{split}
\end{equation*}
\end{theorem}

Given $N_1,N_2$ we choose $k_1,k_2$ such that
$$
N_1^{2(k_1-1)}<p\le N_1^{2k_1},\qquad N_2^{2(k_2-1)}<p\le N_2^{2k_2}
$$
and the bound will be nontrivial unless both $N_1,N_2$ are within $p^{\varepsilon}$-ratio of an element of $\{p^{\frac{1}{2l}},\l\in \Z_{+}\}$.
Thus, we have the following
\begin{cor}
\label{cor:Kloost Karatsuba rang}
Let $I_1=[1, N_1],\, I_2=[1,N_2]$, where for $i=1$ or $i=2$
$$
N_i\not\in \mathop{\bigcup}_{j\ge 1} \, [p^{\frac{1}{2j}-\varepsilon},\, p^{\frac{1}{2j}+\varepsilon}].
$$
Then
$$
\max_{(a,p)=1}\Bigl|\sum_{x_1=1}^{N_1}\sum_{x_2=1}^{N_2}\alpha_1(x_1) \alpha_2(x_2)e_p(ax_1^*x_2^*)\Bigr|<
 p^{-\delta} N_1N_2
$$
for some $\delta=\delta(\varepsilon)>0$.
\end{cor}

\begin{theorem}
\label{thm: Conseq Cor 3}
Let $I_1, I_2\subset \F_p$ be intervals of sizes $N_1,N_2$ in arbitrary position. Then
\begin{equation*}
\begin{split}
\max_{(a,p)=1}\Bigl|\sum_{x_1\in I_1}\sum_{x_2\in I_2}\alpha_1(x_1) \alpha_2(x_2)&e_p(ax_1^*x_2^*)\Bigr|\ll\\
\ll &p^{1/8} N_1^{3/4}N_2^{3/4}\Bigl(\frac{N_1^3}{p}+1\Bigr)^{1/16}\Bigl(\frac{N_2^3}{p}+1\Bigr)^{1/16}.
\end{split}
\end{equation*}
\end{theorem}
Some relevant to Theorem~\ref{thm: Conseq Cor 3} results with intervals starting from the origin can be found in~\cite{Bak},~\cite{FouSh},~\cite{Gar}.

\begin{theorem}
\label{thm:Th1GenBil}
Let $k_1,k_2$ be positive integer constants,  $I_1, I_2\subset \F_p$ be intervals of sizes $N_1,N_2$ in arbitrary position and
$$
N_1<p^{\frac{k_1+1}{2k_1}},\qquad N_2<p^{\frac{k_2+1}{2k_2}}.
$$
Then
\begin{equation*}
\begin{split}
\max_{(a,p)=1}\Bigl|\sum_{x_1\in I_1}\sum_{x_2\in I_2}\alpha_1(x_1)& \alpha_2(x_2)e_p(ax_1^*x_2^*)\Bigr|<\\
&\Bigl(p^{\frac{1}{2k_1k_2}}N_1^{-\frac{1}{k_2(k_1+1)}}N_2^{-\frac{1}{k_1(k_2+1)}}\Bigr)(N_1N_2)^{1+o(1)}.
\end{split}
\end{equation*}
\end{theorem}

We next consider multilinear Kloosterman sums.

\begin{theorem}
\label{thm:Kloost 1/3} Let  $n\ge 7$ and
$N^n>p^{1/3+\varepsilon}$. Then for any intervals $I_1,\ldots, I_n$ of length $N$  we have
$$
\max_{(a,p)=1}\Bigr|\sum_{x_1\in I_1}\ldots\sum_{x_n\in I_n}\alpha_1(x_1)\ldots \alpha_n(x_n)e_p(ax_1^*\ldots x_n^*)\Bigl|<p^{-\delta}N^n
$$
for some $\delta=\delta(\varepsilon,n)>0$.
\end{theorem}

\begin{theorem}
\label{thm:Kloost epsilon} There exists an absolute constant $C>0$ such that for any positive integer $n$ and any intervals $I_1,\ldots, I_n$ of length $N$  with $N>p^{C/n^2}$, we have
$$
\max_{(a,p)=1}\Bigr|\sum_{x_1\in I_1}\ldots\sum_{x_n\in I_n}\alpha_1(x_1)\ldots \alpha_n(x_n)e_p(ax_1^*\ldots x_n^*)\Bigl|<p^{-\delta}N^n
$$
for some $\delta=\delta(n)>0$.
\end{theorem}

\begin{rem}
It can be proved that Theorem~\ref{thm:Kloost epsilon} holds with $C=4$. This can be done  using
the geometry of numbers in the style of~\cite{BGKS2} to get a suitable for this
version of our Theorem~\ref{thm:kI*=kI* small |I|}.
\end{rem}

\begin{theorem}
\label{thm:Kloost 1/2} Let $I_1,\ldots, I_n$ be intervals in $[1,p-1]$ with
$$
|I_1|\cdots|I_n|>p^{1/2+\varepsilon}.
$$
Then we have
$$
\max_{(a,p)=1}\Bigr|\sum_{x_1\in I_1}\ldots\sum_{x_n\in I_n}\alpha_1(x_1)\ldots \alpha_n(x_n)e_p(ax_1^*\ldots x_n^*)\Bigl|<p^{-\delta}|I_1|\cdots|I_n|
$$
for some $\delta=\delta(\varepsilon,n)>0$.
\end{theorem}

There is the following `Archimedian' counterpart of the Karatsuba estimate.

\begin{theorem}
\label{thm:Archimed}
Let $\xi\in\R$ with $|\xi|>N_1N_2$ and $k_1,k_2\in\Z_{+}$. Then
\begin{equation}
\label{eqn:ArchBili}
\Bigl|\sum_{\substack{n_1\sim N_1\\ n_2\sim N_2}}e^{i\frac{1}{n_1}\frac{1}{n_2}\,\xi}\Bigr|<c(k_1,k_2,\varepsilon)\gamma\, (N_1N_2)^{1+\varepsilon}
\end{equation}
with
$$
\gamma=\Bigl\{\Bigl(\frac{|\xi|}{N_1N_2}N_1^{-2k_1}+\frac{N_1N_2}{|\xi|}N_1^{2(k_1-1)}\Bigr)
\Bigl(\frac{|\xi|}{N_1N_2}N_2^{-2k_2}+\frac{N_1N_2}{|\xi|}N_2^{2(k_2-1)}\Bigr)\Bigr\}^{1/(4k_1k_2)}
$$
\end{theorem}

Given $|\xi|>N_1N_2$, choose $k_1,k_2$ satisfying
$$
N_i^{2(k_i-1)}\le \frac{|\xi|}{N_1N_2}<N_i^{2k_i}.
$$
Then each factor in expression for $\gamma$ in Theorem~\ref{thm:Archimed} is $O(1)$.

\smallskip

Exponential sums of the type~\eqref{eqn:ArchBili} appear, for instance, in the proof of Theorem 13.8 in~\cite{FrIw1}
\begin{equation}
\label{eqn:FI pix-pix-y}
\pi(x)-\pi(x-y)\le (2-\delta)\frac{y}{\log y},\quad x^{\theta}<y<x,
\end{equation}
where, as usual, $\pi(z)$ is the number of primes not exceeding $z$ and $\delta=\delta(\theta)>0$. Here $\theta>0$ may be small, $x$ is sufficiently large in terms of $\theta$. In~\cite{FrIw1} the proof of~\eqref{eqn:FI pix-pix-y}
is based on estimates of exponential sums of the form $\sum\limits_{n\sim N}e(\frac{\xi}{n})$ using either Weil or Vinogradov, when $\theta$
is very small. Using Theorem~\ref{thm:Archimed} one gets a better estimate.
\begin{theorem}
\label{thm:pix-pix-y} The estimate~\eqref{eqn:FI pix-pix-y} holds with
$$
\delta<\frac{2(1-\theta)}{12(\theta^{-1}+1)(\theta^{-1}+0.5)+1-\theta}\sim\theta^2.
$$
\end{theorem}

We shall apply trilinear exponential sum bounds from~\cite{B2} (see Lemma~\ref{lem:B2} below) to a linear Kloosterman sums and Brun-Titchmarsh theorem.

\begin{theorem}
\label{thm:linear sqrt log p}
The following bound holds:
$$
\max_{(a,p)=1}\Bigl|\sum_{n\le N}e_p(an^*)\Bigr|\ll \frac{(\log\log p)^3\log p}{(\log N)^{3/2}}\,N,
$$
where the implied constant is absolute.
\end{theorem}

It follows that if $N=p^{\varepsilon}$ with $\varepsilon$ fixed, the saving is $O((\log\log p)^3/(\log p)^{1/2})$ and the estimate is nontrivial
if $N>\exp((\log p)^{\frac{2}{3}}(\log\log p)^3).$
This improves some results of Korolev~\cite{Kor} in the case of prime moduli. We also refer the reader to~\cite{Kor1} for some variants of the problem.

We remark that in~\cite{KarKKR} it is claimed that if $\varepsilon>0$ is fixed, then for $p^{\varepsilon}<N<p^{4/7}$ one has the bound
$$
\Bigl|\sum_{n=1}^{N}e_p(an^*)\Bigr|<\frac{N}{(\log N)^{1-\varepsilon}},
$$
but the proof given there is in doubt.

For $(a,q)=1$, $\pi(x;q,a)$ denotes the number of primes $p\le x, p\equiv a\pmod q.$ We aim to improve the result of Friedlander-Iwaniec on $\pi(x;q,a)$ as follows:

\begin{theorem}
\label{thm:BrunTitch}
Let $q=x^{\theta}$, where $\theta<1$ is close to 1. Then
$$
\pi(x;q,a)<\frac{cx}{\phi(q)\log\frac{x}{q}}
$$
with $c=2-c_1(1-\theta)^2$, for some absolute constant $c_1>0$ and all sufficiently large $x$ in terms of $\theta$.
\end{theorem}

The constant $c_1$ is effective and can be made explicit.

\section{Preliminaries}

Throughout the paper we will use well-known connections between the number of solutions of symmetric equations and
the cardinality of corresponding set. Let $T$ be the number of solutions of the equation
$$
x_1+\ldots+x_n=y_1+\ldots+y_n,
$$
where for each $i$ the variables $x_i,y_i$ run through a set $A_i$. Then for any subset
$$
\Omega\subset A_1\times\ldots\times A_n,
$$
one has the bound
$$
|\{x_1+\ldots+x_n:\,\, (x_1,\ldots x_n)\in\Omega\}|\ge \frac{|\Omega|^2}{T}.
$$
This estimate follows from the observation that  if $T_n(\Omega;\lambda)$
is the number of solutions of the equation
$$
x_1+\ldots+x_n=\lambda,\quad (x_1,\ldots,x_n)\in\Omega,
$$
and
$$
X=\{x_1+\ldots+x_n:\,\, (x_1,\ldots x_n)\in\Omega\},
$$
then
$$
T\ge \sum_{\lambda\in X}T_n(\Omega;\lambda)^2\ge \frac{1}{|X|}\Bigl|\sum_{\lambda\in X}T_n(\Omega; \lambda)\Bigr|^2=\frac{|\Omega|^2}{|X|}.
$$
In particular,
$$
|A_1+\ldots +A_n|=\#\{a_1+\ldots+a_n:\, a_i\in A_i\}\ge\frac{|A_1|^2\ldots |A_n|^2}{T}.
$$

We note that if $A_1,\ldots,A_{2n}\subset \F_p$ and $T_{2n}(\lambda)$ is the number of solutions of the congruence
$$
x_1+\ldots+x_{2n}\equiv \lambda\pmod p,\quad (x_1,\ldots,x_{2n})\in A_1\times\ldots\times A_{2n},
$$
then
$$
T_{2n}(\lambda)\le (J_1\ldots J_{2n})^{\frac{1}{2n}},
$$
where $J_{i}$ is the number of solutions of the congruence
$$
y_1+\ldots+y_n\equiv y_{k+1}+\ldots+y_{2k}\pmod p,\qquad y_1,\ldots,y_{2k}\in A_i.
$$
Indeed, we have
$$
T=\frac{1}{p}\sum_{a=0}^{p-1}\sum_{x_1\in A_1}\ldots\sum_{x_{2n}\in A_{2n}}e_p(ax_1)\ldots e_p(ax_{2n})e_p(-a\lambda).
$$
Applying H\"older's inequality we get
$$
T\le
\prod_{j=1}^n\Bigl(\frac{1}{p}\sum_{a=0}^{p-1}
\Bigl|\sum_{x_j\in A_j}e_p(ax_j)\Bigr|^{2n}\Bigr)^{\frac{1}{2n}}=(J_1\ldots J_{2n})^{\frac{1}{2n}}.
$$

In proofs of some of our results we will use the observation that if $X,Y\in \F_p$, then the number of solutions of the congruence equation
$$
\frac{1}{y+x_1}+\ldots+\frac{1}{y+x_n}= \frac{1}{y+x_{n+1}}+\ldots+\frac{1}{y+x_{2n}},\quad x_i\in X,\, y\in Y,
$$
is at most $O(|X|^nY+|X|^{2n})$, the implied constant may depend only on $n$. Indeed, the contribution from
those $(x_1,\ldots,x_{2n})\in X^{2n}$ for which the series $
x_1,\ldots,x_{2n}$
contains at most $n$ distinct elements, is $O(|X|^n|Y|)$. On the other hand, if there are more than $n$ distinct elements in this series, then we can assume
that $x_1\not\in \{x_2,\ldots,x_{2n}\}$. For each such given $(x_1,\ldots,x_{2n})\in X^{2n}$ the polynomial
$$
P(Z)=\prod_{i\not=1}(Z+x_i)+\ldots+\prod_{i\not=n}(Z+x_i)-\prod_{i\not=n+1}(Z+x_i)-\ldots-\prod_{i\not=2n}(Z+x_i)
$$
is nonzero (as $P(-x_1)\not=0$) and since $P(y)=0$ we get at most $2n-1$ possibilities for $y$. See also~\cite[Lemmas 2,3 ]{Bak} for more
general statements.

\section{Multilinear exponential sums}

The following result, which we state as a lemma, has been proved by Bourgain~\cite{B2}. It is based on results from additive combinatorics,
in particular sum-product estimates. This lemma will be used in the proof of our results  on
Kloosterman sums.
\begin{lemma}
\label{lem:B2}
Let $\gamma_1(x_1),\ldots,\gamma_n(x_n)$ be non-negative real numbers satisfying
$$
\|\gamma_i\|_1=\sum_{x=0}^{p-1}|\gamma_i(x)|\le 1,\quad \|\gamma_i\|_2=\Bigl(\sum_{x=0}^{p-1}|\gamma_i(x)|^2\Bigr)^{1/2}<p^{-\delta}.
$$
Assume further
$$
\prod_{i=1}^{n}\|\gamma_i\|_2<p^{-1/2-\delta},
$$
where $0<\delta<1/4$. Then there is the exponential sum bound
$$
\Bigl|\sum_{x_1=0}^{p-1}\ldots \sum_{x_n=0}^{p-1}\gamma_1(x_1)\ldots\gamma_n(x_n)e_p(x_1\ldots x_n)\Bigr|<p^{-\delta'}
$$
with some $\delta'>(\delta/n)^{Cn}$.
\end{lemma}

\section{Resultant Bound}

We shall need the following resultant bound from~\cite{BGKS2}.

\begin{lemma}
\label{lem:DeterMagic} Let $N\ge1$,
$\sigma,\vartheta\in\R$, and let  $m,n\ge2$ be fixed integers.
Assume also that one of the
following conditions hold:
\begin{itemize}
\item[(i)] $\sigma\ge0$;
\item[(ii)]$\vartheta\ge0$;
\item[(iii)] $\sigma+\vartheta\ge -1$.
\end{itemize}
Let
$P_1(Z)$ and $P_2(Z)$ be non-constant polynomials with integer coefficients,
$$
P_1(Z)=\sum_{i=0}^{m-1}a_{i}Z^{m-1-i}, \qquad P_2(Z)=
\sum_{i=0}^{n-1}b_{i}Z^{n-1-i}
$$
such that
\begin{equation*}
\begin{split}
&|a_{i}|< A N^{i+\sigma}, \quad i =0, \ldots, m-1,\\
&|b_{i}|< A N^{i+\vartheta}, \quad i =0, \ldots, n-1,
\end{split}
\end{equation*}
for some $A$. Then
$$
\Res(P_1, P_2)\ll N^{(m-1+\sigma)(n-1+\vartheta) - \sigma\vartheta.},
$$
where the implicit constant in $\ll$ depends only on $A$, $m$ and $n$.
\end{lemma}

\section{Background on geometry of numbers}

We need some facts from the geometry of numbers. Recall that a lattice in $\R^n$ is an additive subgroup of $\R^n$
generated by $n$ linearly independent vectors. Take an arbitrary
convex compact and symmetric with respect to $0$ body
$D\subset\R^n$. Recall that, for a lattice  $\Gamma\subset\R^n$
and $i=1,\ldots,n$, the $i$-th successive minimum
$\lambda_i(D,\Gamma)$
of the set $D$ with respect to the lattice $\Gamma$ is defined as
the minimal number $\lambda$ such that the set $\lambda D$ contains
$i$ linearly independent vectors of the lattice $\Gamma$. Obviously,
$\lambda_1(D,\Gamma)\le\ldots\le\lambda_n(D,\Gamma)$. We need the
following result given in~\cite[Proposition~2.1]{BHW} (see
also~\cite[Exercise~3.5.6]{TaoVu} for a simplified form that is
still enough for our purposes).

\begin{lemma}
\label{lem:latp} We have
$$
|D\cap\Gamma|\le \prod_{i=1}^n \(\frac{2i}{\lambda_i(D,\Gamma)} + 1\).
$$
\end{lemma}

Denoting, as usual, by $(2n+1)!!$ the product of all odd positive numbers
up to $2n+1$,  we get the following

\begin{cor} \label{cor:latpoints} We have
$$\prod_{i=1}^n \min\{\lambda_i(D,\Gamma),1\} \le \frac{(2n+1)!!}{
|D\cap\Gamma|}.$$
\end{cor}

\section{Equations with many variables}

The following lemma is due to Karatsuba~\cite{Kar1}.

\begin{lemma}
\label{lem:Karatsuba}
The following bound holds:
\begin{equation*}
\begin{split}
\Bigl|\Bigl\{(x_1,\ldots,x_{2k})\in [1,N]^{2k}:&\quad \frac{1}{x_1}+\ldots+\frac{1}{x_k}=\frac{1}{x_{k+1}}+\ldots+\frac{1}{x_{2k}}\Bigr\}\Bigr|\\
< &(2k)^{80k^3}(\log N)^{4k^2}N^k.
\end{split}
\end{equation*}
\end{lemma}

The following elementary statement will be used to exclude some degenerated cases in the proof of Lemma~\ref{lem:symmetricComplexSigma} below.

\begin{lemma}
\label{lem:symmetric}
Let $c\in\C$, $c_1,\ldots,c_{r}\in \C^{*}$, $S$ be a finite subset of $\C$. Let $T_r$ be the number of solutions of the equation
$$
c_1x_1+\ldots+c_rx_r=c,\qquad x_1,\ldots,x_r\in S,
$$
and $J_{2s}$ be the number of solutions of the equation
$$
x_1+\ldots+x_s=x_{s+1}+\ldots+x_{2s}, \qquad x_1,\ldots,x_{2s}\in S.
$$
If $r=2k$ for some integer $k$, then $
T_{2k}\le J_{2k}.$
If $r=2k-1$ for some integer $k$, then $T_{2k-1}^2\le J_{2k-2}J_{2k}.$
\end{lemma}
\begin{proof}
Let $r=2k$. Among all $2k+1$-tuples $(l_1,\ldots,l_{2k}, l)$ with
$$
l_i\in\{\pm c_1,\ldots,\pm c_{2k}\},\quad l\in\{0,c\}
$$
we consider the one for which the number of solutions of the equation
\begin{equation}
\label{eq:symmetcricMaximalL}
l_1x_1+\ldots+l_{2k}x_{2k}=l,\qquad x_1,\ldots,x_{2k}\in S,
\end{equation}
is maximal. There can be several $2k+1$-tuples with this property. We choose the one for which the sequence
$$
l_1,\ldots,l_{2k}
$$
contains the maximal number of elements from $\{-l_1, l_1\}$. We fix one such $(l_1,\ldots,l_{2k}, l)$ with
$$
l_i\in \{-l_1, l_1\},\quad i=1,\ldots,s,
$$
such that either $s=2k$ or $l_t\not\in \{-l_1,l_1\}$ for $t>s$. Denote by  $L_{2k}$ the number of solutions of
\eqref{eq:symmetcricMaximalL}, that is
$$
l_1x_1+\ldots+l_kx_k=l-(l_{k+1}x_{k+1}+\ldots+l_{2k}x_{2k});\qquad x_1,\ldots,x_{2k}\in S.
$$
Note that
$$
L_{2k}=\sum_{\lambda}I_1(\lambda)I_2(\lambda),
$$
where $I_1(\lambda)$ is the number of solutions of the equation
$$
l_1x_1+\ldots+l_kx_k=\lambda,\qquad x_1,\ldots,x_k \in S,
$$
and $I_2(\lambda)$ is the number of solutions of the equation
$$
l-(l_{k+1}x_{k+1}+\ldots+l_{2k}x_{2k})=\lambda;\quad x_{k+1},\ldots,x_{2k}\in S.
$$
Applying the Cauchy-Schwarz inequality, we obtain
$$
L_{2k}^2\le \Bigl(\sum_{\lambda}I_1^2(\lambda)\Bigr)\Bigl(\sum_{\lambda}I_2^2(\lambda)\Bigr).
$$
The quantity in the second parenthesis is equal to the number of solutions of the equation
$$
l_{k+1}x_{1}+\ldots+l_{k}x_{k}=l_{k+1}x_{k+1}+\ldots+l_{2k}x_{2k},\qquad x_1,\ldots,x_{2k}\in S.
$$
Hence, by the maximality of $L_{2k}$ we have
$$
L_{2k}\le \sum_{\lambda}I_1^2(\lambda).
$$
The right hand side indicates the number of solutions of the equation
\begin{equation}
\label{eqn:symmetric l_il_i}
l_1x_1+\ldots+l_kx_k=l_1x_{k+1}+\ldots+l_kx_{2k}, \qquad x_1,\ldots,x_{2k}\in S.
\end{equation}
Clearly, the series
$$
l_1,\ldots,l_k, l_1,\ldots, l_k
$$
contains  $\min\{2s,2k\}$ elements from $\{-l_1,l_1\}$. Hence, by the maximality of $s$ we have $s=2k$. Therefore, $l_i\in \{-l_1,l_1\}$ for all $i$, which implies that the number
of solutions of the equation~\eqref{eqn:symmetric l_il_i} is equal to $J_{2k}$. Thus, $L_{2k}\le J_{2k}$ implying
$T_{2k}\le I_{2k}$. This proves the first statement of the lemma.

To prove the second statement of our lemma, we write the corresponding to $T_{2k-1}$ equation in the form
$$
c_1x_1+\ldots+c_{k}x_{k}=l-(c_{k+1}x_{k+1}+\ldots+c_{2k-1}x_{2k-1})
$$
and, as before, apply the Cauchy-Schwarz inequality to get that
$$
T_{2k-1}^2\le \Bigl(\sum_{\lambda}I_{11}^2(\lambda)\Bigr)\Bigl(\sum_{\lambda}I_{22}^2(\lambda)\Bigr),
$$
where $I_{11}(\lambda)$ is the number of solutions of the equation
$$
c_1x_1+\ldots+c_{k}x_{k}=\lambda,\qquad x_1,\ldots, x_{k} \in S
$$
and $I_{22}(\lambda)$ is the number of solutions of the equation
$$
-(c_{k+1}x_{k+1}+\ldots+c_{2k-1}x_{2k-1})=\lambda,\qquad x_{k+1},\ldots,x_{2k-1}\in S.
$$
Applying the first statement of our lemma we obtain
$$
\sum_{\lambda}I_{11}^2(\lambda)\le J_{2k}; \quad \sum_{\lambda}I_{22}^2(\lambda)\le J_{2k-2},
$$
which finishes the proof of the lemma.
\end{proof}

Let $\xi$ be an algebraic integer of degree $d$ and  $O_\K$ be the ring of integers in $\K=\Q(\xi)$.
In the proof of Lemma~\ref{lem:symmetricComplexSigma} below we use the
language of ideals in the Dedekind domain $O_\K$. We
refer the reader to~\cite[Chapter~12]{IRRO} and~\cite[Chapter~3]{BorShaf} for a background material. Below, all
considered ideals are integral.
In particular, we say that an ideal $I_2$ divides $I_1$ if for some
 ideal $I_3$ we have $I_1=I_2I_3$.

We will use well-known properties of ideals. For instance, if $I_1$ and $I_2$
are  ideals such that $I_1 \subset I_2$ then $I_2$ divides $I_1$
(see, for example,~\cite[Proposition 12.2.7]{IRRO}).

Clearly,  the uniqueness of factorization into prime ideals
implies that if $I_1, I_2, I_3$
are   ideals in $O_\K$ such that $I_3$ divides $I_1 I_2$ then $I_3=J_1 J_2$
for some ideals $J_1$ dividing $I_1$ and $J_2$ dividing $I_2$, respectively.

It is also useful to recall that the number of integral ideals in $O_\K$ of norm
$n$ is at most $\tau(n)^d$, where $\tau$ is the divisor function. In particular,
for fixed constant $d$ and large $n$ this is a quantity of size $n^{o(1)}$.

We recall that the logarithmic height of a nonzero polynomial $P \in
\Z[Z]$ is defined as the maximum
logarithm of the largest (by absolute value) coefficient of $P$.
The logarithmic height of an algebraic number
$\alpha$ is defined as the
logarithmic height of its minimal polynomial.
It is a well-known consequence of basic properties of Mahler's measure that if
$P,Q \in \Z[Z]$ are two univariate non-zero polynomials with $Q\mid P$ and
if $P$ is of logarithmic height at most $H$
then $Q$ is of  logarithmic height at most $H + O(1)$,
where the implied constant depends only on $\deg P$ (see, for example,~\cite[Theorem 4.2.2]{Pras}).

In particular, it follows that if $P\in \Z[Z]$ is a nonconstant polynomial with coefficients
bounded by $M$ (by absolute value), then every root
$\sigma$ of $P(Z)$ can be represented in the form $\xi/q$ where $\xi$ is an algebraic integer
of logarithmic height at most $O(\log M)$ and $q>0$ is an integer with $q<M^{O(1)}$, where the implied constants
depend only on $\deg P.$

\begin{lemma}
\label{lem:symmetricComplexSigma}
For any fixed positive integer constant $r$ and  all values of $\sigma \in \C$
 the number $T_{r}(\sigma,N)$ of solutions
 of the equation
$$
\frac{1}{\sigma+x_1}+\ldots+ \frac{1}{\sigma+x_r}= \frac{1}{\sigma+x_{r+1}}+\ldots+ \frac{1}{\sigma+x_{2r}}
$$
in positive integers $x_1,\ldots,x_{2r}\le N$ satisfies
$$
T_{r}(a,N)<N^{r+o(1)}.
$$
\end{lemma}

\begin{proof}
For the brevity denote $T_r=T_r(a,N)$. We shall prove the lemma by induction on $r$. For $r=1$ the result is trivial.
Let $r\ge 2.$

From Lemma~\ref{lem:symmetric} it follows that the number of solutions satisfying $x_i=x_j$ for some $i\not=j$ contributes to $T_r$ a quantity bounded by
$$
O(\sqrt{T_{r-1}T_{r}})
$$
Thus, by the induction hypothesis it suffices to prove that
$$
T'_{r}<N^{r+o(1)},
$$
where $T'_{r}$ denotes the number of solutions with $x_i\not=x_j$ for all $i\not=j$.
We can assume that $T'_{r}>N^{r}$ as otherwise there is nothing to prove.
Rewrite our equation in the form
$$
\prod_{i\not=1}(\sigma+x_i)+\ldots+\prod_{i\not=r}(\sigma+x_{i})=\prod_{i\not=r+1}(\sigma+x_{i})+\ldots+\prod_{i\not=2r}(\sigma+x_{i})
$$
and consider the polynomial $P(Z)$ defined as
$$
\prod_{i\not=1}(Z+x_i)+\ldots+\prod_{i\not=r}(Z+x_{i})-\prod_{i\not=r+1}(Z+x_{i})-\ldots-\prod_{i\not=2r}(Z+x_{i}).
$$
Clearly, $\deg P\le 2r-1$. Note also that $P(-x_1)\not=0$, so $P(Z)$ is not a zero polynomial. Moreover, $P(\sigma)=0$, implying that $P(Z)$ is not a constant polynomial either. Therefore, we may assume that $\sigma$ is an algebraic number of degree $d$ with $1\le d\le 2r-1$ and logarithmic height $O(\log N)$. We can write $\sigma=\xi/q,$
where $\xi$ is an algebraic integer of height $O(\log N)$ and $q$ is an integer with $q=N^{O(1)}$. Then our equation takes the form
\begin{equation}
\label{eqn:symmetrycAlgebraic}
\prod_{i\not=1}(\xi+qx_i)+\ldots+\prod_{i\not=r}(\xi+qx_i)=\prod_{i\not=r+1}(\xi+qx_i)+\ldots+\prod_{i\not=2r}(\xi+qx_i).
\end{equation}

Let $O_\K$ be the ring of integers in $\K=\Q(\xi)$. The idea is to use the observation that for each $i=1,\ldots, 2r$
\begin{equation}
\label{eqn:symmetryc x_i+qx_j divides}
\xi+qx_i \quad {\rm divides}\quad q^{2r-1}\prod_{j\not=i}(x_j-x_i).
\end{equation}
The strategy to evaluate the number of solutions of~\eqref{eqn:symmetrycAlgebraic} is to introduce consequently the variables $x_1,x_2,\ldots$ taking into account congruence conditions that appeared fixing previous variables.

Given an ideal $I$ we denote by $\nu(I)$ its norm. We factor the principal ideal $(\xi+qx_1)$   into factors
$$
(\xi+qx_1)=I_1J_1,
$$
where the prime factors of $I_1$ divide $q$ and $(J_1,q)=1.$ By~\eqref{eqn:symmetryc x_i+qx_j divides}, the norm $\nu(J_1)$ divides $\prod\limits_{j\ge 2}(x_j-x_1)^d$. Hence,
\begin{equation}
\label{eqn:ideal xj=x1 mod rj}
x_j\equiv x_1\pmod {r_j},\quad 2\le j\le 2r,
\end{equation}
for some $r_j\in\Z_{+}$ such that $\nu(J_1)\mid \nu_1^d$ with $\nu_1=\prod\limits_{j\ge 2}r_j$ and $\nu_1\mid \nu(J_1)$.  We restrict $\nu_1$ to dyadic intervals, that is there exists a fixed number $\mu_1\ge 1$ such that if we restrict $\nu_1$ to the size range
\begin{equation}
\label{eqn:size range nu1}
\mu_1\le \nu_1\le 2\mu_1,
\end{equation}
then the number of solutions of our equation with this restriction will be changed by at most $N^{o(1)}$ times.

For every $x_1$ we consider at most $N^{o(1)}$ different cases and in accordance to~\eqref{eqn:ideal xj=x1 mod rj}
specify $x_j,\, j\ge 2,$ to arithmetic progressions $L_{2,j}\in [1,N]$,
where thus
\begin{equation}
\label{eqn:AlgebraSymmProdj>1}
\prod_{j\ge 2}|L_{2,j}|<\frac{N^{2r-1+o(1)}}{\mu_1}.
\end{equation}
At the next step for $x_2\in L_{2,2}$ we factor
$$
(\xi+qx_2)=I_2J_2,
$$
where the prime factors of $I_2$ divide $q$ or $\nu(\xi+qx_1)$ and $J_2$ is coprime to $q$ and $\nu(\xi+qx_1)$. Hence, $J_2$ is coprime to $(x_2-x_1)$
and again by~\eqref{eqn:symmetryc x_i+qx_j divides}, the norm $\nu(J_2)$ divides $\prod\limits_{j\ge 3}(x_j-x_2)^d$. Arguing as before, we find $\nu_2$ such that $\nu(J_2)\mid \nu_2^d$ and $\nu_2\mid \nu(J_2)$.  We can restrict $\nu_2$ to a size range, that is there exists a fixed number $\mu_2\ge 1$ (independent on variables $x_i$) such that if we restrict $\nu_2$ to
\begin{equation}
\label{eqn:size range nu2}
\mu_2\le \nu_2\le 2\mu_2,
\end{equation}
then the number of solutions of our equation with this restriction will be changed by at most $N^{o(1)}$ times. For every $x_2\in L_{2,2}$ we can consider at most $N^{o(1)}$ possibilities and specify
$x_j,\, j\ge 3,$ to arithmetic progressions $L_{3,j}\subset L_{2,j}$ where
\begin{equation}
\label{eqn:AlgebraSymmProdj>2}
\prod_{j\ge 3}|L_{3,j}|\le \frac{1}{\mu_2}\prod_{j\ge 3}|L_{2,j}|,\qquad \nu(J_2)\mid \nu_2^{d}.
\end{equation}
Note indeed that the progression $L_{2,j}$ are defined to some modulus $r_j\mid \nu_1$ and $\gcd(r_j,\nu(J_2))=1$ since $\nu(J_1),\, \nu(J_2)$ are coprime.

At the next step for $x_3\in L_{3,3}$ we factor
$$
(\xi+qx_3)=I_3J_3,
$$
where the prime factors of $I_3$ divides either $(q),\nu(\xi+qx_1)$ or $\nu(\xi+qx_2)$, and $J_3$ is coprime with $(q),\nu(\xi+qx_1)$ and $\nu(\xi+qx_2)$.
We find $\nu_3$ similar to the previous cases and specify it to a size range $\mu_3\le \nu_3\le 2\mu_3$, where $\mu_3$ is independent on variables. The continuation of the process is clear.

We now consequently fix $x_1\le N $, $x_2\in L_{2,2},\ldots,x_{2r}\in L_{2r,2r}$, that is, we fix $x_1$ and considering $N^{o(1)}$ possibilities for arithmetic progressions $L_{2,j},j\ge 3,$ we fix $x_2\in L_{2,2}$, then  considering $N^{o(1)}$ possibilities for arithmetic progressions $L_{3,j}, j\ge 3,$ we fix $x_3\in L_{3,3}$ and iterate this until we fix $x_{2r}\in L_{2r,2r}$. We estimate the number of solutions of~\eqref{eqn:symmetrycAlgebraic} as $N^{o(1)}$ contributions of the form
$$
N|L_{2,2}||L_{3,3}|\ldots|L_{2r,2r}|.
$$
From~\eqref{eqn:AlgebraSymmProdj>1},~\eqref{eqn:AlgebraSymmProdj>2} and iteration we get
\begin{equation*}
\begin{split}
N^{2r-1}&\gtrapprox \mu_1 |L_{2,2}|\prod_{j\ge 3}|L_{2,j}|\\
&\gtrapprox \mu_1\mu_2 |L_{2,2}||L_{3,3}|\prod_{j\ge 4}|L_{3,j}|\\
&\qquad\qquad\ldots\\
&\gtrapprox \mu_1\ldots\mu_{2r-1}|L_{2,2}|\ldots |L_{2r,2r}|.
\end{split}
\end{equation*}
Here $A\gtrapprox B$ means $A>BN^{o(1)}$.
Thus, the number of solutions of~\eqref{eqn:symmetrycAlgebraic} may be bounded by
\begin{equation}
\label{eqn:AlgebraicBounded1}
\frac{N^{2r+o(1)}}{\mu_1\ldots \mu_{2r-1}}.
\end{equation}

Next, returning to our construction, it is clear that $\nu(I_1)=N^{O(1)}$ and since the prime factors of $I_1$ divide $q$,
it follows that the number of possibilities for $I_1$ is at most $N^{o(1)}$. Fixing $I_1$ and denoting
$\xi_1=\xi,\xi_2,\ldots,\xi_d$  the conjugates of $\xi$, we have
$$
\prod_{s=1}^d(\xi_s+qx_1)=\nu(I_1)\nu(J_1)
$$
and since $\nu(J_1)\mid \nu_1^d$ it follows that $\nu(J_1)$ is determined by $\nu_1$ with up to $N^{o(1)}$  possibilities. Thus, given $\nu_1$ we retrieve $x_1$
with up to $N^{o(1)}$ possibilities. It follows that in the size range~\eqref{eqn:size range nu1} the number of possibilities for $x_1$  is at most $N^{o(1)}\mu_1$.
Next, once $x_1$ is given, there are at most $N^{o(1)}$ possibilities for the ideal $I_2$ and similarly $N^{o(1)}\mu_2$ possibilities for $x_2$.
It follows that the number of possibilities for $x_1,x_2,\ldots,x_{2r-1}$ is at most $\mu_1\mu_2\ldots\mu_{2r-1}N^{o(1)}$. Thus, the number of solutions
of~\eqref{eqn:symmetrycAlgebraic} is bounded by $\mu_1\mu_2\ldots\mu_{2r-1}N^{o(1)}$. Since it is also bounded by~\eqref{eqn:AlgebraicBounded1},  the result follows.
\end{proof}

\begin{lemma}
\label{lem:Identity}
Let $x,y,z,a_1,a_2,b_1,b_2$ be complex numbers such that
$$
\left\{\begin{array}{llll}
xyz=a_1(x+y+z)+b_1,\\
xy+yz+zx=a_2(x+y+z)+b_2,
\end{array}
\right.
$$
Then
$$
(x^2-a_2x+a_1)(y^2-a_2y+a_1)(z^2-a_2z+a_1)=(b_1-\alpha_1b_2-\alpha_1^3)(b_1-\alpha_2b_2-\alpha_2^3),
$$
where
$$
\alpha_1=\frac{a_2+\sqrt{a_2^2-4a_1}}{2},\quad \alpha_2=\frac{a_2-\sqrt{a_2^2-4a_1}}{2}.
$$
\end{lemma}

\begin{proof}
Indeed, since $\alpha_i^2-a_2\alpha_i+a_1=0$, we have
$$
(x-\alpha_i)(y-\alpha_i)(z-\alpha_i)=b_1-\alpha_i b_2-\alpha_i^3,\quad i=1,2.
$$
Multiplying these equalities the claim follows.
\end{proof}

\begin{lemma}
\label{lem:LinearDioph}
Let $A, B$ be integers with $AB\not=0$ and
$|A|, |B|<N^{O(1)}.$
Then the diophantine equation
\begin{equation}
\label{eqn:LinearDioph}
Axy+Bx+By=0
\end{equation}
has at most $N^{o(1)}$ solutions in integers $x,y$ with $|x|,|y|\le N^{O(1)}.$
\end{lemma}

\begin{proof} Indeed, we have
$$
(Ax+B)(Ay+B)=B^2
$$
and the statement follows from the well-known bound for the divisor function.
\end{proof}

\begin{lemma}
\label{lem:dioph 3I*}
Let $a_0,b_0,u_0,v_0$ be integers with $b_0u_0v_0\not=0$ and
$$
|a_0|, |b_0|, |u_0|, |v_0|< N^{O(1)}.
$$
Assume that
$$
\frac{u_0}{v_0}\not\in\Bigl\{\frac{b_0}{a_0+b_0x}:\quad 1\le x\le N\Bigr\}.
$$
Then the number $J$ of solutions of the diophantine equation
\begin{equation*}
\begin{split}
&u_0(a_0+b_0x_1)(a_0+b_0x_2)(a_0+b_0x_3)=v_0b_0\times\\
&\Bigl((a_0+b_0x_1)(a_0+b_0x_2)+(a_0+b_0x_2)(a_0+b_0x_3)+(a_0+b_0x_3)(a_0+b_0x_1)\Bigr)
\end{split}
\end{equation*}
in integers $x_1,x_2,x_3$ with
$$
1\le x_i\le N,\quad a_0+b_0x_i\not=0,
$$
satisfies
$$
J<N^{2/3+o(1)}.
$$
\end{lemma}

\begin{proof}
We can clearly assume that
$$
b_0>0, \quad v_0>0,\quad \gcd(a_0,b_0)=1.
$$
We observe that if one of the variables $x_1,x_2,x_3$ is determined, then for the rest two variables there remain at most $N^{o(1)}$
possibilities. Indeed, let $x_1$ be fixed. Then denoting
$$
X_i=a_0+b_0x_i,\quad A=u_0X_1-v_0b_0, \quad B=-v_0b_0X_1,
$$
we get
$$
AX_2X_3+BX_2+BX_3=0.
$$
By the condition, $AB\not=0$ and $|A|, |B|< N^{O(1)}$. Hence by Lemma~\ref{lem:LinearDioph}, we can retrieve $X_2,X_3$, and thus the numbers $x_2,x_3$, with at most $N^{o(1)}$ possibilities.

Denote
$$
u_0'=\frac{u_0}{\gcd(u_0, b_0v_0)}, \qquad \frac{b_0v_0}{\gcd(u_0, b_0v_0)}=w_0\ge 1.
$$

We have
\begin{equation*}
\begin{split}
&u_0'(a_0+b_0x_1)(a_0+b_0x_2)(a_0+b_0x_3)=w_0\times\\
&\Bigl((a_0+b_0x_1)(a_0+b_0x_2)+(a_0+b_0x_2)(a_0+b_0x_3)+(a_0+b_0x_3)(a_0+b_0x_1)\Bigr)
\end{split}
\end{equation*}

Since $(u_0',w_0)=1,$ there exists a representation
\begin{equation}
\label{eqn:w_0=w_1w_2w_3}
w_0=w_1w_2w_3
\end{equation}
and non-zero integers $y_1,y_2,y_3$ such that
\begin{equation}
\label{eqn:a_0+b_0x_i=w_iy_i}
a_0+b_0x_i=w_iy_i,\quad i=1,2,3.
\end{equation}
In particular,
\begin{equation}
\label{eqn:x_itoy_i}
u_0'y_1y_2y_3=w_1w_2y_1y_2+w_2w_3y_2y_3+w_3w_1y_3y_1.
\end{equation}
We can assume that
$$
w_1\ge w_2\ge w_3\ge 1.
$$
By the bound for the divisor function, the representation~\eqref{eqn:w_0=w_1w_2w_3} implies that there are at most $N^{o(1)}$ possible values for $w_1, w_2, w_3.$ Let us fix one such representation.  Having $w_1,w_2,w_3$ fixed, we observe that the condition $\gcd(a_0,b_0)=1$ and the equality
\begin{equation}
\label{eqn:a_0+b_0x_1=w_1y_1}
a_0+b_0x_1=w_1y_1
\end{equation}
imply that $\gcd(b_0,w_1)=1$. Hence,~\eqref{eqn:a_0+b_0x_1=w_1y_1} uniquely determines $x_1\pmod {w_1}$. It then follows  that there are at most
$$
N^{1+o(1)}w_1^{-1}+1
$$
possible values for $x_1$. Then we retrieve $x_2,x_3$ and get the bound
\begin{equation}
\label{eqn:J and 1/w_1}
J<N^{1+o(1)}w_1^{-1}+N^{o(1)}.
\end{equation}

Next, from~\eqref{eqn:x_itoy_i} and $w_1\ge w_2\ge w_3\ge 1$ we have
$$
|u_0'|\min\{|y_1|,|y_2|,|y_3|\}\le 3 w_1^2.
$$
Thus,
$$
\min\{|y_1|,|y_2|,|y_3|\}\le 3w_1^{2}.
$$
Hence, we can determine one of $y_1,y_2,y_3$ with $O(w_1^2)$ possibilities. Consequently, by~\eqref{eqn:a_0+b_0x_i=w_iy_i} we determine one of $x_1,x_2,x_3$ and thus we get
$$
J<w_1^2N^{o(1)}.
$$
Comparing this with~\eqref{eqn:J and 1/w_1}, we conclude $J<N^{2/3+o(1)}$.

\end{proof}

\section{Congruences}

In what follows, $N$ is a large parameter, $N<p$. We start with the following result from~\cite{CillGar} which is based on the idea of Heath-Brown~\cite{HB}.
\begin{lemma}
\label{lem:CillGar}
Let $\lambda\not\equiv 0\pmod p$. Then the number $J$ of solutions of the congruence
$$
xy\equiv \lambda\pmod p,\qquad L+1\le x, y\le L+N,
$$
satisfies
$$
J< \frac{N^{3/2+o(1)}}{p^{1/2}}+N^{o(1)}.
$$
In particular, if $N<p^{1/3}$, then one has $J<N^{o(1)}$.
\end{lemma}

\begin{corollary}
\label{cor:CillGar}
Let $\lambda\not\equiv 0\pmod p$. Then the number $J$ of solutions of the congruence
$$
\frac{1}{x}+\frac{1}{y}\equiv \lambda\pmod p,\qquad L+1\le x, y\le L+N,
$$
satisfies
$$
J< \frac{N^{3/2+o(1)}}{p^{1/2}}+N^{o(1)}.
$$
In particular, if $N<p^{1/3}$, then one has $J<N^{o(1)}$.
\end{corollary}
\begin{proof}
Indeed, we have
$$
(x-\lambda^{-1})(y-\lambda^{-1})\equiv\lambda^{-2}\pmod p
$$
and the claim follows from Lemma~\ref{lem:CillGar}.
\end{proof}

The following result has been proved in~\cite{CillGar} (see also~\cite{BGKS1} for the extension of~\cite{CillGar}
to higher dimensional case).
\begin{lemma}
\label{lem:CillGarBGKS}
Let $\lambda\not\equiv 0\pmod p$ and $N<p^{1/8}$. Then the number $J$ of solutions of the congruence
$$
xyz\equiv \lambda\pmod p,\qquad L+1\le x, y\le L+N,
$$
satisfies
$$
J< N^{o(1)}.
$$
\end{lemma}

The following lemma follows from the work of Ayyad, Cochrane and Zheng~\cite{ACZ} (and from Lemma~\ref{lem:CillGar} when $|I_1||I_2|$ is very small).
In fact, we shall only apply this lemma when one of the intervals starts from the origin, the result which had previously been established by Friedlander and Iwaniec~\cite{FrIw0}.

\begin{lemma}
\label{lem:multEnergyTwoInt}
Let $I_1, I_2$ be two intervals in $\F_p^*$ with
$$
|I_1||I_2|<p.
$$
Then the number of solutions of the congruence
$$
xy=zt,\quad (x,z)\in I_1\times I_1,\quad (y,t)\in I_2\times I_2,
$$
is not greater than $(|I_1||I_2|)^{1+o(1)}$.
\end{lemma}

The following lemma will be used in the proof of Theorems~\ref{thm:3I* small |I|} and~\ref{thm:3I*=3I* small |I|}. It is given
with explicit constants to make the statement more transparent, the reader should not take them
seriously.

\begin{lemma}
\label{lem:3I* small |I|} Let $I=\{a+1,\ldots, a+N\}$ and
$$
\lambda\not\equiv 0\pmod p,\quad \lambda\not\in \{x^{-1}\pmod p:\quad x\in I\}.
$$
Assume that
$$
|I|=N<0.1p^{1/18}J^{2/9}
$$
where $J$ is the number of solutions of the congruence
$$
x^{-1}+y^{-1}+z^{-1}\equiv \lambda\pmod p,\quad x,y,z\in I.
$$
Let also $J>N^{\varepsilon}$ for some fixed small constant $\varepsilon>0$ and let $N$ be sufficiently large.
Then there exist integers $\Delta_4',\Delta_4{''},\Delta_3$ with
$$
|\Delta_4'|<10^5 N^4/J, \quad |\Delta_4{''}|<10^5 N^4/J,\quad |\Delta_3|<10^5 N^3/J
$$
such that
$$
a\equiv \frac{\Delta_4'}{\Delta_3}\pmod p;\quad  \lambda^{-1}\equiv \frac{\Delta_4^{''}}{\Delta_3}\pmod p.
$$
\end{lemma}

\begin{proof}
 It follows that  $J$ is the number of solutions of the congruence
\begin{equation*}
\begin{split}
\lambda(a+x)&(a+y)(a+z)\\
&\equiv (a+x)(a+y)+(a+y)(a+z)+(a+z)(a+x)\pmod p
\end{split}
\end{equation*}
in positive integers $x,y,z\le N$ with $(a+x)(a+y)(a+z)\not\equiv 0\pmod p$.

Note that by Corollary~\ref{cor:CillGar} we have $J< N^{1+o(1)}$ so that  $N< p^{1/13}$. We rewrite the congruence in the form
\begin{equation}
\label{eqn:lambdaxyz}
\begin{split}
xyz & +(a-\lambda^{-1})(xy+yz+zx)+\\ &(a^2-2a\lambda^{-1})(x+y+z)+(a^3-3a^2\lambda^{-1})
\equiv 0\pmod p.
\end{split}
\end{equation}
We fix one solution $(x_0,y_0,z_0)$ and get
\begin{equation}
\label{eqn:xyz}
\begin{split}
(a^2-&2a\lambda^{-1})(x+y+z-A_0)\\+&(a-\lambda^{-1})(xy+yz+zx-B_0) +(xyz-C_0)\equiv 0\pmod p.
\end{split}
\end{equation}
where
$$
A_0=x_0+y_0+z_0,\quad B_0=x_0y_0+y_0z_0+z_0x_0, \quad C_0= x_0y_0z_0
$$
We use some ideas from~\cite{BGKS2}. Define the lattice
$$
\Gamma = \{(u,v,w)\in\Z^3~:~(a^2-2a\lambda^{-1})u+(a-\lambda^{-1})v+w\equiv 0 \pmod p\}
$$
and the body
$$D = \{(u,v,w)\in\R^3~:~|u|\le 3 N,\,|v|\le 3N^2,\,|w|\le N^3\}.$$
Since any given vector
$$
(x+y+z-A_0, xy+yz+zx-B_0, xyz-C_0)
$$
defines the values of $x,y,z$ with at most $6$ possibilities, we have
$$|D\cap\Gamma|\ge J/6.$$
Therefore, by Corollary~\ref{cor:latpoints}, the successive minimas
$\lambda_i=\lambda_i(D,\Gamma)$, $i=1,2,3$, satisfy the inequality
$$
\prod_{i=1}^3\min\{1,\lambda_i\}<1000 J^{-1}.
$$
In particular, we have $\lambda_1\le 1$. By the definition of $\lambda_i$, there are linearly
independent vectors
$$
(u_i,v_i,w_i)\in\lambda_iD\cap\Gamma,\qquad i=1,2,3.
$$
We consider separately the following three cases.

\bigskip

{\it Case~1\/}: $\lambda_3\le 1$.
Thus,  we have $\lambda_1\lambda_2\lambda_3<1000J^{-1}$.
We consider the determinant
\begin{equation*}
\Delta = \det  \(
  \begin{array}{cccccccc}
    u_1 & v_1 & w_1\\
    u_2 & v_2 & w_2\\
    u_3 & v_3 & w_3\\
  \end{array}
\).
\end{equation*}
Clearly,
$$
|\Delta|< 6N^6\lambda_1\lambda_2\lambda_3 <6000 N^6/J<p.
$$
Thus, $|\Delta |<p.$
On the other hand,
from
$$
(a^2-2a\lambda^{-1})u_i+(a-\lambda^{-1})v_i+w_i\equiv 0\pmod p,\qquad i=1,2,3,
$$
we conclude that $\Delta$ is divisible by $p$. Therefore,
$\Delta=0$,
which contradicts the linear independence of the vectors
$(u_i,v_i,w_i)$, $i=1,2,3$. Thus, this case is impossible.

\bigskip

{\it Case~2\/}: $\lambda_1\le 1,\lambda_2>1$.
Since $\lambda_2>1$ we see that
$$
(x+y+z-A_0, xy+yz+zx-B_0, xyz-C_0)
$$
and $(u_1,v_1,w_1)$ are
linearly dependent. Therefore, one of the two conditions hold:
\begin{itemize}
\item[(i)] $x+y+z-A_0=0$;
\item[(ii)]
$
\left\{\begin{array}{llll}
xyz=\frac{w_1}{u_1}(x+y+z)+C_0-\frac{w_1}{u_1}A_0,\\
xy+yz+zx=\frac{v_1}{u_1}(x+y+z)+B_0-\frac{v_1}{u_1}A_0,
\end{array}
\right.
$
\end{itemize}
If a solution $(x,y,z)$ satisfy (i), then our congruence~\eqref{eqn:lambdaxyz} can be written in the form
$$
\Bigl(x+a-\frac{1}{\lambda}\Bigr)\Bigl(y+a-\frac{1}{\lambda}\Bigr)\Bigl(z+a-\frac{1}{\lambda}\Bigr)\equiv \lambda'\pmod p.
$$
Since $\lambda\not \in I^{-1}\pmod p$, we have $\lambda'\not\equiv 0\pmod p.$ Hence, by Lemma~\ref{lem:CillGarBGKS} the solutions counted in (i) contributes to our $J$ at most the quantity $N^{o(1)}$.

If a solution $(x,y,z)$ satisfy (ii), then Lemma~\ref{lem:Identity} and the bound for the divisor function implies that one of the variables $x,y,z$ is determined with at most $N^{o(1)}$ possibilities. Therefore, by Corollary~\ref{cor:CillGar} we get that the solutions counted on (ii)
contributes to our $J$ also at most the quantity $N^{o(1)}$.

Thus, we get $J<N^{o(1)}$ contradicting our assumption that $J>N^{\varepsilon}$. Therefore, Case 2 is impossible.

\bigskip

{\it Case~3\/}: $\lambda_1\le 1,\lambda_2\le 1,\lambda_3>1$.

Thus,  $\lambda_1\lambda_2<1000 J^{-1}$. Next,
\begin{equation}
\label{eq:case2nu=3}
\(\begin{array}{cc}
    u_1&v_1\\
    u_2&v_2\\
  \end{array}\)\(
  \begin{array}{c}
   a^2-2a\lambda^{-1}\\
  a-\lambda^{-1}\\
  \end{array}
\)\equiv \(
  \begin{array}{c}
   -w_1\\
  -w_2\\
  \end{array}
\)\pmod p.
\end{equation}
Let
$$
\Delta_3 =\det \(\begin{array}{cc}
    u_1&v_1\\
    u_2&v_2\\
  \end{array}\), \, \Delta_5=\det \(\begin{array}{cc}
    -w_1&v_1\\
    -w_2&v_2\\
  \end{array}\),\, \Delta_4=\det \(\begin{array}{cc}
    u_1&-w_1\\
    u_2&-w_2\\
  \end{array}\).
$$

We have
\begin{equation}
\label{eq:case2nu=3determinants}
|\Delta_3|<2000 N^3/J,\qquad |\Delta_5|<2000 N^5/J,\qquad |\Delta_4|<2000 N^4/J.
\end{equation}

We observe that
\begin{equation}
\label{eq:Det nozero}
\Delta_3\not\equiv 0\pmod p.
\end{equation}
Indeed, assuming the contrary, from the congruence~\eqref{eq:case2nu=3} we get
$$
\Delta_3\equiv \Delta_5\equiv \Delta_4\equiv 0\pmod p.
$$
Taking into account~\eqref{eq:case2nu=3determinants}, this implies that
$$
\Delta_3= \Delta_5= \Delta_4= 0.
$$
It then follows that the rank of the matrix
$$
\(\begin{array}{ccc}
    u_1&v_1&w_1\\
    u_2&v_2&w_2\\
  \end{array}\)
$$
is strictly less than $2$, which contradicts the linear independence of the vectors $(u_i,v_i,w_i)$, $i=1,2$.
Thus, we have~\eqref{eq:Det nozero}. Hence,
$$
a^2-2a\lambda^{-1}\equiv \frac{\Delta_5}{\Delta_3}\pmod p, \quad a-\lambda^{-1} \equiv \frac{\Delta_4}{\Delta_3}\pmod p.
$$
Using~\eqref{eqn:lambdaxyz}, we also have
$$
a^3-3a^2\lambda^{-1}\equiv \frac{\Delta_6}{\Delta_3}\pmod p,
$$
for some integer $\Delta_6$ with
$$
|\Delta_6|<6000 N^6/J.
$$
We have
\begin{equation}
\label{eqn:lambda -2}
\lambda^{-2}\equiv \frac{\Delta_4^2}{\Delta_3^2}-\frac{\Delta_5}{\Delta_3}\equiv \frac{\Delta_4^2-\Delta_5 \Delta_3}{\Delta_3^2}\pmod p.
\end{equation}
Furthermore, substituting  $\lambda^{-1}\equiv a-(\Delta_4/\Delta_3)$ in the other two equations we get
$$
\Delta_3a^2-2\Delta_4 a+\Delta_5\equiv 0\pmod p,
$$
$$
2\Delta_3 a^3-3\Delta_4a^2+\Delta_6\equiv 0\pmod p.
$$
It follows that
$$
\Delta_4 a^2-2\Delta_5 a+\Delta_6\equiv 0\pmod p.
$$
Consider the polynomials
$$
P(Z)=
\Delta_4 Z^2-2\Delta_5Z+\Delta_6, \quad Q(Z)=
\Delta_3Z^2-2\Delta_4 Z+\Delta_5.
$$
Since
$$
P(a)\equiv Q(a)\equiv 0\pmod p,
$$
we have
$$
\Res(P,Q)\equiv 0\pmod p.
$$
On the other hand
$$
|\Res(P,Q)|<10^{18} N^{18}/J^{4}<p.
$$
Thus,
$$
\Res(P,Q)=0.
$$
It then follows that
 the polynomials $P(Z)$ and $Q(Z)$ have a common root. If $Q(Z)$ is irreducible in $\Q[\Z]$, then $Q(Z)$ and $P(Z)$ are linearly dependent, implying that
$$
\Delta_4^2=\Delta_5 \Delta_3.
$$
In view of~\eqref{eqn:lambda -2} this is impossible.
Hence, $Q(Z)$ is irreducible in $\Q[\Z]$, and therefore its discriminant is a square of an integer. Thus,
$$
\Delta_4^2-\Delta_5 \Delta_3=m^2,\quad m\in \Z_{+}.
$$
It then follows that $m<10^4 N^4/J$. Furthermore,
$$
\lambda^{-2}\equiv \frac{m^2}{\Delta_3^2}\pmod p.
$$
Hence,
$$
\lambda^{-1}\equiv\frac{\Delta_4'}{\Delta_3}\pmod p,\quad |\Delta_4'|=m<10^4 N^4/J.
$$
Consequently
$$
a\equiv \frac{\Delta_4}{\Delta_3}+\lambda^{-1}\equiv \frac{\Delta_4{''}}{\Delta_3}\pmod p,\quad |\Delta_4{''}|<10^5 N^4/J.
$$

\end{proof}

\section{Proof of Theorems~\ref{thm:kI*}--\ref{thm:kI*=kI* I=[1,N]}}

\subsection{Proof of Theorem~\ref{thm:kI*}}
Let $I=[a+1,a+N]$.  We first consider the case $N<p^{\frac{k+1}{2k}}$. Thus, we are aiming to prove that in this case one has the bound
$$
J_{2k}<N^{2k^2/(k+1)+o(1)}.
$$
Put
$$
V=[N^{(k-1)/(k+1)}];\quad Y=[N^{2/(k+1)}];\quad I_1=[a+1,a+2N].
$$
First we define an appropriate subset $\cV\in [0.5V,\, V]$. For a given $v\in [0.5V, \, V]$ define the function
$\eta_v:I_1\to\Z_{+}$ by
$$
\eta_v(u')=\Bigl|\Bigl\{(u_1',v_1)\in I_1\times [0.5V,\, V];\quad u'v_1\equiv u_1'v\pmod p\Bigr\}\Bigr|.
$$
By Lemma~\ref{lem:multEnergyTwoInt},
$$
\sum_{u'\in I_1}\,\,\sum_{v\in [0.5V,\, V]}\eta_v(u')<N^{1+o(1)}V.
$$
Therefore, there is a subset $\cV\in [0.5V,\, V]$ with $|\cV|\sim V$ such that
\begin{equation}
\label{eqn:Th1 sum eta}
\sum_{u'\in I_1}\eta_v(u')<N^{1+o(1)}\quad {\rm for \,\,\,any}\quad v\in\cV.
\end{equation}

For any fixed integers $y_i\in\cY=[0.5Y,\, Y]$ and $v\in \cV$ the quantity $J_{2k}$ does not exceed the number of solutions of the congruence
$$
\frac{1}{u_1'-vy_1}+\ldots+\frac{1}{u_k'-vy_k}\equiv \frac{1}{u_{k+1}'-vy_{k+1}}+\ldots+\frac{1}{u_{2k}'-vy_{2k}}\pmod p
$$
in integers $u_i'\in I_1$. Thus, summing up over $y_i$ and $v\in \cV$ we get
\begin{equation}
\label{eqn:YZR<trig sum}
Y^{2k}V J_{2k}\ll \frac{1}{p}\sum_{n=0}^{p-1}\sum_{v\in\cV}\Bigl|\sum_{u'\in I_1}\sum_{y\in\cY}e_p(n(u'-vy)^{-1})\Bigr|^{2k}.
\end{equation}

For $v\in \cV$ and $B$ of the form $B=2^s<N$, denote
$$
S_{v,B}=\Bigl\{u'\in I_1;\quad 0.5B\le\eta_v(u')<B\Bigr\}.
$$
Hence
$$
I_1=\bigcup_BS_{v,B}
$$
and since $v\in\cV$, by~\eqref{eqn:Th1 sum eta}
\begin{equation}
\label{eqn:Th1SvB<}
|S_{v,B}|<\frac{N^{1+o(1)}}{B}.
\end{equation}
From~\eqref{eqn:YZR<trig sum} we clearly have
$$
Y^{2k}V J_{2k}\ll \frac{N^{o(1)}}{p}\sum_B\sum_{n=0}^{p-1}\sum_{v\in\cV}\Bigl(\sum_{u'\in S_{v,B}}\Bigl|\sum_{y\in\cY}e_p(n(u'-vy)^{-1})\Bigl|\Bigr)^{2k}.
$$
Hence, by the H\"older inequality and~\eqref{eqn:Th1SvB<}, we get, for some fixed $B$,
\begin{equation}
\label{eqn:Th1 passing to SvB}
Y^{2k}V J_{2k}\ll \frac{N^{o(1)}}{p}\Bigl(\frac{N}{B}\Bigr)^{2k-1}\,\sum_{n=0}^{p-1}\sum_{v\in\cV}\sum_{u'\in S_{v,B}}\Bigl|\sum_{y\in\cY}e_p(n(u'-vy)^{-1})\Bigl|^{2k}.
\end{equation}
The quantity
$$
\frac{1}{p}\sum_{n=0}^{p-1}\sum_{v\in\cV}\sum_{u'\in S_{v,B}}\Bigl|\sum_{y\in\cY}e_p(n(u'-vy)^{-1})\Bigl|^{2k}
$$
is bounded by the number of solutions of the congruence
$$
\frac{1}{u'-vy_1}+\ldots+\frac{1}{u'-vy_k}\equiv \frac{1}{u'-vy_{k+1}}+\ldots+\frac{1}{u'-vy_{2k}}\pmod p
$$
in variables $v\in \cV, u'\in S_{v,B}, y\in \cY$. Thus, we obtain the bound
$$
\frac{1}{p}\sum_{n=0}^{p-1}\sum_{v\in\cV}\sum_{u'\in S_{v,B}}\Bigl|\sum_{y\in\cY}e_p(n(u'-vy)^{-1})\Bigl|^{2k}\ll Y^{k}V\frac{N^{1+o(1)}}{B}+Y^{2k}B.
$$
Hence, from~\eqref{eqn:Th1 passing to SvB} we get
$$
Y^{2k}V J_{2k}\ll N^{2k-1+o(1)}B^{-2k+1}\Bigl(Y^kV\frac{N^{1+o(1)}}{B}+Y^{2k}B\Bigr)
$$
Thus,
$$
J_{2k}<N^{2k+o(1)}Y^{-k}+\frac{N^{2k-1+o(1)}}{V}<N^{\frac{2k^2}{k+1}+o(1)},
$$
which proves the result in the case $N<p^{\frac{k+1}{2k}}$.

Let now $N>p^{\frac{k+1}{2k}}$. We split the interval $I=[a+1,a+N]$ into $K\sim Np^{-\frac{k+1}{2k}}$ subintervals of length at most $N_1=p^{\frac{k+1}{2k}}$. Thus, for some intervals $I^{(1)},\ldots, I^{(2k)}$
of length $N_1$ we have the bound
$$
J_{2k}<K^{2k}R_{2k}\ll \Bigl(\frac{N}{p^{\frac{k+1}{2k}}}\Bigr)^{2k}R_{2k},
$$
where $R_{2k}$ is the number of solutions of the congruence
$$
\frac{1}{x_1}+\ldots+\frac{1}{x_k}\equiv \frac{1}{x_{k+1}}+\ldots+\frac{1}{x_{2k}}\pmod p,\qquad x_i\in I^{(i)},\, i=1,\ldots,2k.
$$
Expressing the number of solutions of this congruence in terms of exponential sums, applying the H\"{o}lder inequality we get that
$$
R_{2k}\le\prod_{i=1}^{2k}(R_{2k}(i))^{1/2k},
$$
where $R_{2k}(i)$ is the number of solutions of the congruence
$$
\frac{1}{x_1}+\ldots+\frac{1}{x_k}\equiv \frac{1}{x_{k+1}}+\ldots+\frac{1}{x_{2k}}\pmod p,\qquad x_1,\ldots,x_{2k}\in I^{(i)}.
$$
Thus, for some fixed $i=i_0$ one has
$$
R_{2k}<R_{2k}(i_0).
$$
Since $|I^{(i_0)}|<N_1=p^{\frac{k+1}{2k}}$, we already know that
$$
R_{2k}(i_0)< N_1^{2k^2/(k+1)+o(1)}=p^{k}N^{o(1)}.
$$
Thus,
$$
J_{2k}<\Bigl(\frac{N}{p^{\frac{k+1}{2k}}}\Bigr)^{2k}p^{k}N^{o(1)}=\frac{N^{2k+o(1)}}{p},
$$
which concludes the proof of Theorem~\ref{thm:kI*}.

\subsection{Proof of Theorem~\ref{thm:3I* small |I|}}

Let $I=\{a+1,\ldots,a+N\}$. We can assume that $N$ is large and $J>N^{2/3}\log N$, as otherwise there is nothing to prove. The conditions of Lemma~\ref{lem:3I* small |I|} are satisfied,
so that there exist integers $\Delta_4',\Delta_4{''}$ and $\Delta_3$ with
$$
|\Delta_4'|< N^{10/3}, \quad |\Delta_4{''}|< N^{10/3},\quad |\Delta_3|< N^{7/3}
$$
such that
$$
a\equiv \frac{\Delta_4'}{\Delta_3}\pmod p;\quad  \lambda^{-1}\equiv \frac{\Delta_4^{''}}{\Delta_3}\pmod p.
$$
Substituting this in
\begin{equation*}
\begin{split}
\lambda&(a+x)(a+y)(a+z)\\
&\equiv (a+x)(a+y)+(a+y)(a+z)+(a+z)(a+x)\pmod p
\end{split}
\end{equation*}
we obtain
\begin{equation*}
\begin{split}
(\Delta_4'+&\Delta_3x)(\Delta_4'+\Delta_3y)(\Delta_4'+\Delta_3z)\equiv \Delta_4''\Bigl\{(\Delta_4'+\Delta_3x)(\Delta_4'+\Delta_3y)\\
&+(\Delta_4'+\Delta_3y)(\Delta_4'+\Delta_3z)+(\Delta_4'+\Delta_3z)(\Delta_4'+\Delta_3x)\Bigr\}\pmod p.
\end{split}
\end{equation*}
The left and the right hand sides are of the order of magnitude $O(N^{10})=o(p)$. Thus, the congruence is converted to the equality
\begin{equation*}
\begin{split}
(\Delta_4'+&\Delta_3x)(\Delta_4'+\Delta_3y)(\Delta_4'+\Delta_3z)=\Delta_4''\Bigl\{(\Delta_4'+\Delta_3x)(\Delta_4'+\Delta_3y)\\
&+(\Delta_4'+\Delta_3y)(\Delta_4'+\Delta_3z)+(\Delta_4'+\Delta_3z)(\Delta_4'+\Delta_3x)\Bigr\}
\end{split}
\end{equation*}
and the claim follows from Lemma~\ref{lem:dioph 3I*}.

\subsection{Proof of Theorem~\ref{thm:3I*=3I* small |I|}}

Let $I=\{a+1,\ldots,a+N\}$. The statement is equivalent to the claim that for any $\varepsilon>0$ one has the bound
$$
J_6\ll N^{3+\varepsilon},
$$
where the implied constant may depend only $\varepsilon$.

Observe that for any $j\in\Z$ there are $u_{j}, v_{j}\in \Z$
such that
\begin{equation}
\label{eqn:j=uj/vj}
\frac{u_{j}}{v_{j}}\equiv j\pmod p;\quad |u_{j}|\le p^{1/2},\quad 0<|v_{j}|\le p^{1/2}.
\end{equation}
This follows from the fact that among  more than $p$ numbers
$$
u+j v,\quad 0\le u,v\le [p^{1/2}]
$$
there are at least two numbers congruent modulo $p$. We also represent $a$ in this form, that is
\begin{equation}
\label{eqn:a=a0/b0}
a\equiv\frac{a_0}{b_0}\pmod p;\quad |a_0|\le p^{1/2},\quad 0<|b_0|\le p^{1/2}.
\end{equation}
Let $T_{j}$ be the number of solutions of the congruence
$$
\frac{1}{a+x_1}+\frac{1}{a+x_2}+\frac{1}{a+x_3}\equiv j\pmod p;\quad 1\le x_1,x_2,x_3\le N.
$$
We have
$$
J_6=\sum_{j=0}^{p-1}T_j^{2}
$$
and
$$
\sum_{j=0}^{p-1}T_j\le N^3
$$
From Corollary~\ref{cor:CillGar} it clearly follows that $T_j< N^{1+o(1)}$.
Therefore, it follows that the contribution to $J_6$ from
$
j\in \{I^{-1}\cup {0}\} \pmod p
$
is
$$
\sum_{\substack{0\le j\le p-1\\ j\in I^{-1}\cup 0 \pmod p}}T_j^2\le (N+1)\max_{j}T_J^2<N^{3+o(1)}.
$$
Furthermore, the
contribution to $J_6$ from those $j$ for which
$T_j<N^{0.1\varepsilon}$
is less than
$$
N^{0.1\varepsilon}\sum_{j=0}^{p-1}T_j\le N^{3+0.1\varepsilon}.
$$
Thus, if we denote by $\Omega$ the set of integers $j$ with
$$
1\le j\le p-1,\quad j\not\in I^{-1}\pmod p,\quad |T_j|>N^{0.1\varepsilon},
$$
then
\begin{equation}
\label{eqn:J6<sum j Omega}
J_6<N^{3+0.2\varepsilon}+\sum_{j\in \Omega}T_j^2.
\end{equation}
Now for each $j$ we apply Lemma~\ref{lem:3I* small |I|} (where $J$ is substituted by $T_j$ and $\lambda$ by $j$.
Then there exist numbers $\Delta_{4j}',\Delta_{4j}'',\Delta_{3j}$ with
$$
\Delta_{4j}'\ll N^4, \quad \Delta_{4j}{''}\ll N^4,\quad \Delta_{3j}\ll N^3
$$
such that
$$
a\equiv \frac{\Delta_{4j}'}{\Delta_{3j}}\pmod p;\qquad  j^{-1}\equiv \frac{\Delta_{4j}''}{\Delta_{3j}}\pmod p.
$$
Comparing this with~\eqref{eqn:j=uj/vj} and~\eqref{eqn:a=a0/b0}, we see that
\begin{equation}
\label{eqn:a=a0/b0=Delta4jDelta3j mod p}
a\equiv \frac{a_0}{b_0}\equiv \frac{\Delta_{4j}'}{\Delta_{3j}}\pmod p;\qquad j^{-1}\equiv \frac{v_j}{u_j}\equiv \frac{\Delta_{4j}''}{\Delta_{3j}}\pmod p.
\end{equation}
Taking into account the inequality conditions on $N,a_0,b_0,u_j,v_j$, we see that we have equality
\begin{equation}
\label{eqn:a0/b0=Delta4jDelta3j}
\frac{a_0}{b_0}= \frac{\Delta_{4j}'}{\Delta_{3j}};\qquad \frac{v_j}{u_j}= \frac{\Delta_{4j}''}{\Delta_{3j}}.
\end{equation}
Now we represent the congruence corresponding to $T_j$ in the form
\begin{equation*}
\begin{split}
x_1x_2x_3 & +(a-j^{-1})(x_1x_2+x_2x_3+x_3x_1)\\ & +(a^2-2aj^{-1})(x_1+x_2+x_3)+(a^3-3a^2j^{-1})
\equiv 0\pmod p.
\end{split}
\end{equation*}
Using~\eqref{eqn:a=a0/b0=Delta4jDelta3j mod p}, we substitute $a$ and $j^{-1}$, implying
\begin{equation*}
\begin{split}
x_1x_2x_3 & +\Bigl(\frac{\Delta_{4j}'}{\Delta_{3j}}-\frac{\Delta_{4j}''}{\Delta_{3j}}\Bigr)(x_1x_2+x_2x_3+x_3x_1)\\ & +\Bigl(\Bigl(\frac{\Delta_{4j}'}{\Delta_{3j}}\Bigr)^2-2\frac{\Delta_{4j}'}{\Delta_{3j}}\cdot\frac{\Delta_{4j}''}{\Delta_{3j}}\Bigr)(x_1+x_2+x_3)\\
&+ \Bigl(\Bigl(\frac{\Delta_{4j}'}{\Delta_{3j}}\Bigr)^3-3\Bigl(\frac{\Delta_{4j}'}{\Delta_{3j}}\Bigr)^2\cdot \frac{\Delta_{4j}''}{\Delta_{3j}}\Bigr)
\equiv 0\pmod p.
\end{split}
\end{equation*}
After multiplying by $\Delta_{3j}^3$ the left hand side becomes an integer of the size $O(N^{12})=o(p).$
Thus, the resulting congruence is converted to the equality, and dividing by $\Delta_{3j}^3$ we get
\begin{equation*}
\begin{split}
x_1x_2x_3 & +\Bigl(\frac{\Delta_{4j}'}{\Delta_{3j}}-\frac{\Delta_{4j}''}{\Delta_{3j}}\Bigr)(x_1x_2+x_2x_3+x_3x_1)\\ &+\Bigl(\Bigl(\frac{\Delta_{4j}'}{\Delta_{3j}}\Bigr)^2-2\frac{\Delta_{4j}'}{\Delta_{3j}}\cdot\frac{\Delta_{4j}''}{\Delta_{3j}}\Bigr)(x_1+x_2+x_3)\\
&+\Bigl(\Bigl(\frac{\Delta_{4j}'}{\Delta_{3j}}\Bigr)^3-3\Bigl(\frac{\Delta_{4j}'}{\Delta_{3j}}\Bigr)^2\cdot \frac{\Delta_{4j}''}{\Delta_{3j}}\Bigr)
= 0.
\end{split}
\end{equation*}
We use~\eqref{eqn:a0/b0=Delta4jDelta3j} and write this equality in the form
\begin{equation*}
\begin{split}
x_1x_2x_3 & +\Bigl(\frac{a_0}{b_0}-\frac{v_j}{u_j}\Bigr)(x_1x_2+x_2x_3+x_3x_1)\\ &+\Bigl(\frac{a_0^2}{b_0^2}-2\frac{a_0}{b_0}\cdot\frac{v_j}{u_j}\Bigr)(x_1+x_2+x_3)\\
&+ \Bigl(\frac{a_0^3}{b_0^3}-3\frac{a_0^2}{b_0^2}\cdot \frac{v_j}{u_j}\Bigr)
= 0.
\end{split}
\end{equation*}
Consequently we get
\begin{equation}
\label{eqn:k=3 equality}
\frac{1}{(a_0/b_0)+x_1}+\frac{1}{(a_0/b_0)+x_2}+\frac{1}{(a_0/b_0)+x_3}=\frac{u_j}{v_j}.
\end{equation}
Thus, for $j\in \Omega$ the quantity $T_j$ is just the number of solutions of the diophantine equation~\eqref{eqn:k=3 equality} in positive integers $x_1,x_2,x_3\le N$. Since $u_j/v_j$ are pairwise distinct, we get that the quantity $\sum\limits_{j\in \Omega}T_j^2$
is not greater than the number of solutions of the equation
\begin{equation*}
\begin{split}
\frac{1}{(a_0/b_0)+x_1}+&\frac{1}{(a_0/b_0)+x_2}+\frac{1}{(a_0/b_0)+x_3}\\
&=\frac{1}{(a_0/b_0)+x_4}+\frac{1}{(a_0/b_0)+x_5}+\frac{1}{(a_0/b_0)+x_6}
\end{split}
\end{equation*}
in positive integers $x_1,\ldots,x_6\le N.$ Therefore, by Lemma~\ref{lem:symmetricComplexSigma} we get that
$$
\sum_{j\in \Omega}T_j^2<N^{3+o(1)}.
$$
In view of~\eqref{eqn:J6<sum j Omega}, this completes the proof of our theorem.

\subsection{Proof of Theorem~\ref{thm:kI*=kI* small |I|}}

Let $I=\{a+1,\ldots,a+N\}$. Using standard arguments involving H\"older's inequality (combined with inductive process), it suffices to show that
the contribution from the set of solutions with pairwise $x_1,\ldots,x_{2k}$ is $N^{k+o(1)}$. Thus, in what follows, we consider $x_1,\ldots,x_{2k}$
pairwise distinct.

For each solution $\vec{x} = (x_1,\ldots,x_{2k})$,
we  consider the polynomial $P_{\vec{x}}(Z)$ defined as
$$
\prod_{i\not=1}(Z+x_i)+\ldots+\prod_{i\not=k}(Z+x_{i})-\prod_{i\not=k+1}(Z+x_{i})-\ldots-\prod_{i\not=2k}(Z+x_{i}).
$$
Clearly, $\deg P_{\vec{x}}(Z)\le 2k-2$. Note also that $P_{\vec{x}}(-x_1)\not= 0$. In particular, $P_{\vec{x}}(Z)$ is not a zero polynomial, and since $P_{\vec{x}}(a)\equiv 0\pmod p$, it is not a constant polynomial neither. Clearly, $P_{\vec{x}}(Z)$ has the form
$$
P_{\vec{x}}(Z)=\sum_{i=0}^{2k-2}a_iZ^{2k-2-i}
$$
where $|a_i|\ll N^{i+1}$.
We fix one solution
$$
(x_1,\ldots,x_{2k})=(c_1,\ldots,c_{2k})
$$
and consider the polynomial $P_{\vec{c}}(Z)$ that corresponds to $(c_1,\ldots,c_{2k})$. Since
$$
P_{\vec{x}}(a)\equiv P_{\vec{c}}(a)\equiv 0\pmod p,
$$
we get that
$$
\Res(P_{\vec{x}}(Z),P_{\vec{c}}(Z))\equiv 0\pmod p.
$$
On the other hand, $P_{\vec{x}}(Z)$ and $P_{\vec{c}}(Z)$ satisfy the condition of Lemma~\ref{lem:DeterMagic} with
$$
\sigma=\theta=1,\qquad m=n=2k-1.
$$
Hence,
$$
\Res(P_{\vec{x}}(Z),P_{\vec{c}}(Z))\ll N^{4k^2-4k}.
$$
Therefore, assuming $N<p^{1/(4k^2)}$, we get
$$
|\Res(P_{\vec{x}}(Z),P_{\vec{c}}(Z))|<p,
$$
whence
$$
\Res(P_{\vec{x}}(Z),P_{\vec{c}}(Z))=0.
$$
It follows that for every solution $\vec{x}=(x_1,\ldots,x_{2k})$ the polynomial  $P_{\vec{x}}(Z)$ has a common root with $P_{\vec{c}}(Z)$. Since $x_i$ are pairwise distinct, the condition $P_{\vec{x}}(\sigma)=0$ implies that $x_i+\sigma\not=0$. Thus, Lemma~\ref{lem:symmetricComplexSigma} implies that for every root $\sigma$ of $P_{\vec{c}}(Z)$ the equation
$P_{\vec{x}}(\sigma)=0$
has at most $N^{k+o(1)}$ solutions in positive integers $x_i\le N$.  The claim now follows.

\subsection{Proof of Theorems~\ref{thm:kI*=kI* I=[1,N]} and~\ref{thm:kI*=kI* I=[1,N] prime}}

First we prove Theorem~\ref{thm:kI*=kI* I=[1,N]}. It suffices to consider the case $kN^{k-1}<p$ as otherwise the statement is trivial. For $\lambda=0,1,\ldots,p-1$  denote
$$
J(\lambda)=\Bigl\{(x_1,\ldots,x_k)\in I^k:\quad x_1^*+\ldots+x_k^*\equiv \lambda\pmod p\Bigr\}.
$$
Let
$$
\Omega=\{\lambda\in [1, p-1]:\quad |J(\lambda)|\ge 1\}.
$$
Since $J(0)=0$, we have
$$
J_{2k}=\sum_{\lambda\in \Omega}|J(\lambda)|^2.
$$
Consider the lattice
$$
\Gamma_{\lambda}=\{(u,v)\in \Z^2:\quad \lambda u\equiv v\pmod p\}
$$
and the body
$$
D=\{(u,v)\in \R^2:\quad |u|\le N^k,\,\, |v|\le kN^{k-1}\}.
$$
Denoting by $\mu_1,\mu_2$ the consecutive minimas of the body $D$ with respect to the lattice
$\Gamma_{\lambda}$, by Corollary~\ref{cor:latpoints} it follows
$$
\prod_{i=1}^2\min\{\mu_i, 1\}\le \frac{15}{|\Gamma_{\lambda}\cap D|}.
$$
Observe that for $(x_1,\ldots,x_k)\in J(\lambda)$ one has
$$
\lambda x_1\ldots x_k\equiv x_2\ldots x_{k}+\ldots+x_1\ldots x_{k-1}\pmod p,
$$
implying
$$
(x_1\ldots x_k,\,  x_2\ldots x_{k}+\ldots+x_1\ldots x_{k-1})\in \Gamma_{\lambda}\cap D.
$$
Thus, for $\lambda\in \Omega$ we have $\mu_1\le 1$. We split the set  $\Omega$
into two subsets:
$$
\Omega'=\{\lambda\in \Omega:\quad \mu_2\le 1\},\qquad
\Omega''=\{\lambda\in \Omega:\quad \mu_2> 1\}.
$$
We have
\begin{equation}
\label{eqn:Kloost Kar Range J2k S epsilon}
\begin{split}
J_{2k} = \sum_{\lambda\in \Omega'}|J(\lambda)|^2+\sum_{\lambda\in \Omega''}|J(\lambda)|^2.
\end{split}
\end{equation}

\bigskip

{\it Case~1\/}:  $\lambda\in \Omega'$, that is $\mu_2\le 1$. Let $(u_i,v_i)\in \mu_i D\cap\Gamma_{\lambda},\, i=1,2,$ be linearly independent.
Then
$$
0\not=\det \(\begin{array}{cc}
    u_1&v_1\\
    u_2&v_2\\
  \end{array}\)\equiv 0\pmod p,
$$
whence
$$
\Bigl|\det \(\begin{array}{cc}
    u_1&v_1\\
    u_2&v_2\\
  \end{array}\)\Bigr|\ge p.
$$
Also
$$
\Bigl|\det \(\begin{array}{cc}
    u_1&v_1\\
    u_2&v_2\\
  \end{array}\)\Bigr|\le 2k\mu_1\mu_2 N^{2k-1}\le \frac{30kN^{2k-1}}{|\Gamma\cap D|}.
$$
Thus, for $\lambda\in \Omega'$,  the number $|\Gamma_{\lambda}\cap D|$ of solutions of the congruence
$$
\lambda u\equiv v\pmod p
$$
in integers $u,v$ with $|u|\le N^k,\, |v|\le kN^{k-1}$ is bounded by
\begin{equation}
\label{eqn:Gammalambda cap D le}
|\Gamma_{\lambda}\cap D|\le \frac{30kN^{2k-1}}{p}.
\end{equation}
Therefore, if we denote by $S(u,v)$ the set of $k$-tuples $(x_1,\ldots,x_k)$ of positive integers $x_1,\ldots,x_k\le N$ with
$$
x_1\ldots x_k=u,\quad x_2\ldots x_{k}+\ldots+x_1\ldots x_{k-1}= v,
$$
we get
$$
\sum_{\lambda\in\Omega'}|J(\lambda)|^2=\sum_{\lambda\in\Omega'}\Bigl(\sum_{(u,v)\in\Gamma_{\lambda}\cap D}\,\,\sum_{(x_1,\ldots,x_k)\in S(u,v)}1\Bigr)^2.
$$
Applying the Cauchy-Schwarz inequality and taking into account~\eqref{eqn:Gammalambda cap D le}, we get
$$
\sum_{\lambda\in\Omega'}|J(\lambda)|^2=\frac{30kN^{2k-1}}{p}\sum_{\lambda\in\Omega'}\sum_{(u,v)\in\Gamma_{\lambda}\cap D}\Bigl(\sum_{(x_1,\ldots,x_k)\in S(u,v)}1\Bigr)^2
$$
The summation on the right hand side is clearly bounded by the number of solutions of the system of equations
$$
\left\{\begin{array}{llll}
x_1\ldots x_k=y_1\ldots y_k,\\
x_1\ldots x_{k-1}+\ldots+x_2\ldots x_{k}=y_2\ldots y_{k}+\ldots+y_1\ldots y_{k-1},
\end{array}
\right.
$$
in positive integers $x_i,y_j\le N.$ Hence, by Lemma~\ref{lem:Karatsuba}, it follows that
\begin{equation}
\label{eqn:sum over Omega prima}
\sum_{\lambda\in\Omega'}|J(\lambda)|^2<30k(2k)^{80k^3}(\log N)^{4k^2}\frac{N^{3k-1}}{p}.
\end{equation}

\bigskip

{\it Case~2\/}:  $\lambda\in \Omega''$, that is $\mu_2 > 1$. Then the vectors from $\Gamma\cap D$ are linearly dependent and
in particular there is some $\widehat{\lambda}\in\Q$ such that
$$
\widehat{\lambda}x_1\ldots x_k=x_2\ldots x_{k}+\ldots+x_1\ldots x_{k-1} \quad {\rm for} \quad (x_1,\ldots,x_k)\in J(\lambda).
$$
Thus,
\begin{equation*}
\begin{split}
&\sum_{\lambda\in \Omega''}|J(\lambda)|^2\le \sum_{\widehat{\lambda}\in \Q}\Bigl|\{(x_1,\ldots,x_k)\in I^k:\, \frac{1}{x_1}+\ldots+\frac{1}{x_k}=\widehat{\lambda}\Bigr|^2\\=
&\Bigl|\Bigl\{(x_1,\ldots,x_{2k})\in [1,N]^{2k}: \quad \frac{1}{x_1}+\ldots+\frac{1}{x_k}=\frac{1}{x_{k+1}}+\ldots+\frac{1}{x_{2k}}\Bigr\}\Bigr|\\
< &(2k)^{80k^3}(\log N)^{4k^2}N^k.
\end{split}
\end{equation*}
Inserting this and~\eqref{eqn:sum over Omega prima} into~\eqref{eqn:Kloost Kar Range J2k S epsilon}, we obtain
\begin{equation*}
\begin{split}
J_{2k}<(2k)^{90k^3}(\log N)^{4k^2}\Bigl(\frac{N^{2k-1}}{p}+1\Bigr)N^{k}
\end{split}
\end{equation*}
which concludes the proof of Theorem~\ref{thm:kI*=kI* I=[1,N]}.

The proof of Theorem~\ref{thm:kI*=kI* I=[1,N] prime} follows the same line  with the only difference
that instead of Lemma~\ref{lem:Karatsuba} one should apply the bound
\begin{equation*}
\begin{split}
\Bigl|\Bigl\{(x_1,\ldots,x_{2k})\in ([1,N]\cap \cP)^{2k}:&\quad \frac{1}{x_1}+\ldots+\frac{1}{x_k}=\frac{1}{x_{k+1}}+\ldots+\frac{1}{x_{2k}}\Bigr\}\Bigr|\\
< &(2k)^{k}\Bigl(\frac{N}{\log N}\Bigr)^k.
\end{split}
\end{equation*}

\section{Proof of Theorems~\ref{thm:Kloost double 18/37}--\ref{thm:Kloost 1/2}}

\subsection{Proof of Theorem~\ref{thm:Kloost double 18/37}}

It suffices to deal with the case  $|I_1|=[p^{1/18}],\, |I_2|=[p^{5/12+\varepsilon}]$ and $\varepsilon<0.1$. Let
$$
W_2=\sum_{x_1\in I_1}\sum_{x_2\in I_2}\alpha_1(x_1)\alpha_2(x_2)e_p(ax_1^*x_2^*).
$$
We take $k=[1/\varepsilon]$ and apply the H\"older inequality;
\begin{equation*}
\begin{split}
|W_2|^{k}\le & |I_1|^{k-1}\sum_{x_1\in I_1}\Bigl|\sum_{x_2\in I_2}\alpha_2(x_2)e_p(ax_1^*x_2^*)\Bigr|^k\\=
&|I_1|^{k-1}\sum_{x_1\in I_1}\Bigl|\sum_{y_1,\ldots,y_k\in I_2}\alpha_2(y_1)\ldots\alpha_2(y_k)e_p(ax_1^*(y_1^*+\ldots+y_k^*))\Bigr|\\=
&|I_1|^{k-1}\sum_{x_1\in I_1}\theta(x_1)\Bigl(\sum_{y_1,\ldots,y_k\in I_2}\alpha_2(y_1)\ldots\alpha_2(y_k)e_p(ax_1^*(y_1^*+\ldots+y_k^*))\Bigr)\\
\le &|I_1|^{k-1} \sum_{y_1,\ldots,y_k\in I_2}\Bigl|\sum_{x_1\in I_1}\theta(x_1)e_p(ax_1^*(y_1^*+\ldots+y_k^*))\Bigr|,
\end{split}
\end{equation*}
where $\theta(x)$ some complex numbers with $|\theta(x)|\le 1$. We again apply the H\"older inequality and obtain
\begin{equation*}
\begin{split}
|W_2|^{3k}\le & |I_1|^{3k-3} |I_2|^{2k}\sum_{y_1,\ldots,y_k\in I_2}\Bigl|\sum_{x_1\in I_1}\theta(x_1)e_p(ax_1^*(y_1^*+\ldots+y_k^*))\Bigr|^3\\=
& |I_1|^{3k-3} |I_2|^{2k}\sum_{\lambda=0}^{p-1}T(\lambda)\Bigl|\sum_{x_1\in I_1}\theta(x_1)e_p(ax_1^*\lambda)\Bigr|^3,
\end{split}
\end{equation*}
where $T(\lambda)$ is the number of solutions of the congruence
$$
y_1^{*}+\ldots+y_k^{*}\equiv \lambda\pmod p,\quad y_i\in I_2.
$$
We apply now the Cauchy-Schwarz inequality and get
\begin{equation*}
\begin{split}
|W_2|^{6k}\le & |I_1|^{6k-6} |I_2|^{4k}\Bigl(\sum_{\lambda=0}^{p-1}T(\lambda)^2\Bigr)\sum_{\lambda=0}^{p-1}\Bigl|\sum_{x_1\in I_1}\theta(x_1)e_p(ax_1^*\lambda)\Bigr|^6.
\end{split}
\end{equation*}
By Theorem~\ref{thm:kI*}, we have
$$
\sum_{\lambda=0}^{p-1}T(\lambda)^2<|I_2|^{2k-2+1/(k+1)+o(1)}<|I_2|^{2k-2+\varepsilon}.
$$
Furthermore,
$$
\sum_{\lambda=0}^{p-1}\Bigl|\sum_{x_1\in I_1}\theta(x_1)e_p(ax_1^*\lambda)\Bigr|^6\le pJ_6,
$$
where $J_6$ is the number of solutions of the congruence
$$
y_1^*+y_2^*+y_3^*\equiv y_4^*+y_5^*+y_6^*\pmod p,\quad y_i\in I_1.
$$
By Theorem~\ref{thm:3I*=3I* small |I|}, we have $J_6=|I_1|^{3+o(1)}$. Thus,
$$
|W_2|^{6k}\le p|I_1|^{6k-3+o(1)}|I_2|^{6k-2+\varepsilon}.
$$
Since $p\ll |I_1|^3|I_2|^2p^{-2\varepsilon}$, the result follows.

\subsection{Proof of Theorems~\ref{thm:Kloost Karatsuba range},~\ref{thm: Conseq Cor 3},~\ref{thm:Th1GenBil}}

Let
$$
S=\sum_{x_1\in I_1}\sum_{x_2\in I_2}\alpha_1(x_1) \alpha_2(x_2)e_p(ax_1^*x_2^*).
$$
Then by H\"{o}lder's inequality
$$
|S|^{k_2}\le N_1^{k_2-1}\sum_{x_1\in I_1}\Bigl|\sum_{x_2\in I_2} \alpha_2(x_2)e_p(ax_1^*x_2^*)\Bigr|^{k_2}.
$$
Thus, for some $\sigma(x_1)\in\C,\, |\sigma(x_1)|=1$,
$$
|S|^{k_2}\le N_1^{k_2-1}\sum_{y_1,\ldots,y_{k_2}\in I_2}\Bigl|\sum_{x_1\in I_1} \sigma(x_1)e_p(ax_1^*(y_1^*+\ldots+y_{k_2}^*)\Bigr|.
$$
Again by H\"{o}lder's inequality,
$$
|S|^{k_1k_2}\le N_1^{k_1k_2-k_1}N_2^{k_1k_2-k_2}\sum_{\lambda=0}^{p-1}J_{k_2}(\lambda;N_2)\Bigl|\sum_{x_1\in I_1} \sigma(x_1)e_p(ax_1^*\lambda\Bigr|^{k_1},
$$
where $J_k(\lambda; N)$ is the number of solutions of the congruence
$$
x_1^*+\ldots+x_k^*\equiv \lambda\pmod p,\qquad x_i\in [1,N].
$$
Then applying the Cauchy-Schwarz inequality and using
$$
\sum_{\lambda=0}^{p-1}J_{k_2}(\lambda; N_2)^2=J_{2k_2}(N_2),\qquad \sum_{\lambda=0}^{p-1}\Bigl|\sum_{x_1\in I_1} \sigma(x_1)e_p(ax_1^*\lambda\Bigr|^{2k_1}\le
pJ_{2k_1}(N_1).
$$
we get
\begin{equation}
\label{eqn:th91011}
|S|^{2k_1k_2}\le pN_1^{2k_1k_2-2k_1}N_2^{2k_1k_2-2k_2}J_{2k_1}(N_1)J_{2k_2}(N_2).
\end{equation}
Applying Theorem~\ref{thm:kI*=kI* I=[1,N]}, we obtain
\begin{equation*}
\begin{split}
|S|^{2k_1k_2}\le &(2k_1)^{90k_1^3}(2k_2)^{90k_2^3}(\log N_1)^{4k_1^2}(\log N_2)^{4k_2^2}\times \\
\times &N_1^{2k_1k_2}N_2^{2k_1k_2}\Bigl(\frac{N_1^{k_1-1}}{p^{1/2}}+\frac{p^{1/2}}{N^{k_1}}\Bigr)
\Bigl(\frac{N_2^{k_2-1}}{p^{1/2}}+\frac{p^{1/2}}{N^{k_2}}\Bigr).
\end{split}
\end{equation*}
Thus,
\begin{equation*}
\begin{split}
|S|< & (2k_1)^{45k_1^2/k_2}(2k_2)^{45k_2^2/k_1}(\log p)^{2(\frac{k_1}{k_2}+\frac{k_2}{k_1})}\times \\
\times &\Bigl(\frac{N_1^{k_1-1}}{p^{1/2}}+\frac{p^{1/2}}{N^{k_1}}\Bigr)^{1/(2k_1k_2)}
\Bigl(\frac{N_2^{k_2-1}}{p^{1/2}}+\frac{p^{1/2}}{N^{k_2}}\Bigr)^{1/(2k_1k_2)}N_1N_2,
\end{split}
\end{equation*}
which finishes the proof of Theorem~\ref{thm:Kloost Karatsuba range}.

To prove Theorem~\ref{thm: Conseq Cor 3}, we note that if $1\le N<p,$ then as a consequence of Corollary~\ref{cor:CillGar}, the number of solutions of the congruence
$$
\frac{1}{x_1}+\frac{1}{x_2}\equiv \frac{1}{x_3}+\frac{1}{x_4}\pmod p,\quad L+1\le x_1,x_2,x_3,x_4\le L+N
$$
is bounded by $N^{2+o(1)}(N^{3/2}p^{-1/2}+1)$. Following the proof of Theorem~\ref{thm:Kloost Karatsuba range} with $k_1=k_2=2$ and applying this bound with $[L+1,L+N]=I_i$, we derive  Theorem~\ref{thm: Conseq Cor 3}.

To prove Theorem~\ref{thm:Th1GenBil} we use~\eqref{eqn:th91011}, where in this case $J_{2k_i}(N_i)$ is the number of solutions of the congruence
$$
x_1^{-1}+\ldots+x_k^{-1}=x_{k+1}^{-1}+\ldots+x_{2k}^{-1},\qquad (x_1,\ldots,x_{2k})\in I_i^{2k}.
$$
Since
$$
N_1<p^{\frac{k_1+1}{2k_1}},\qquad N_2<p^{\frac{k_2+1}{2k_2}},
$$
by Theorem~\ref{thm:kI*} we have
$$
J_{2k_1}(N_1)<N_1^{\frac{2k_1^2}{k_1+1}},\qquad J_{2k_2}(N_2)<N_2^{\frac{2k_2^2}{k_2+1}}.
$$
Incorporating this in~\eqref{eqn:th91011}, the result follows.

\subsection{Proof of Theorem~\ref{thm:Kloost 1/3}}

Denote
$$
W_n=\Bigr|\sum_{x_1\in I_1}\ldots\sum_{x_n\in I_n}\alpha_1(x_1)\ldots \alpha_n(x_n)e_p(ax_1^*\ldots x_n^*)\Bigl|.
$$
Applying $n$ times the H\"older inequality and using that for $|\alpha(v)|\le 1$ one has
$$
\sum_{u}\Bigl|\sum_{v}\alpha(v)e_p(auv)\Bigr|^6\le \sum_{v_1,\ldots,v_6}\Bigl|\sum_{u}e_p(a(v_1+v_2+v_3-v_4-v_5-v_6)u)\Bigr|,
$$
it follows that
\begin{equation*}
\begin{split}
W_n^{6^n}&\le N^{n6^n-6n}\times\\
\times &\sum_{x_{11},\ldots, x_{16}\in I_1}\ldots \sum_{x_{n1},\ldots x_{n6}\in I_n}e_p(a(x_{11}^*+\ldots-x_{16}^*)\ldots(x_{n1}^*+\ldots-x_{n6}^*)).
\end{split}
\end{equation*}
We can fix $x_{j4},x_{j5},x_{j6}$ such that for some  integers $c_j$
$$
W_n^{6^n}\le N^{n6^n-3n}
\Bigl|\sum_{x_{11}, x_{12}, x_{13}\in I_1}\ldots \sum_{x_{n1},x_{n2},x_{n3}\in I_n}\{...\}\Bigr|
$$
where in the brackets $\{...\}$ we have the expression
$$
e_p(a(x_{11}^*+x_{12}^*+x_{13}^*-c_1)\ldots(x_{n1}^*+x_{n2}^*+x_{n3}^*-c_n)).
$$
Let $T_j(\lambda)$, $j=1,\ldots,n$, be the number of solutions of the congruence
$$
x_1^*+x_2^*+x_3^*-c_j\equiv\lambda\pmod p,\quad x_1,x_2,x_3\in I_j.
$$
Then we get
\begin{equation}
\label{eqn:W_n 6^n}
W_n^{6^n}\le N^{n6^n-3n}
\Bigl|\sum_{\lambda_1=0}^{p-1}\ldots \sum_{\lambda_n=0}^{p-1}T_1(\lambda_1)\ldots T_{n}(\lambda_n)e_p(a\lambda_1\ldots\lambda_n)\Bigr|.
\end{equation}
Now we observe that
$$
\sum_{\lambda=0}^{p-1}\frac{T_j(\lambda)}{N^3}\le 1
$$
and also by Theorem~\ref{thm:3I*=3I* small |I|} we have
$$
\sum_{\lambda}\Bigl(\frac{T_j(\lambda)}{N^3}\Bigr)^2<N^{-3+o(1)}.
$$
Furthermore,
$$
\prod_{j=1}^{n} \Bigl(\sum_{\lambda}\Bigl(\frac{T_j(\lambda)}{N^3}\Bigr)^2\Bigr)^{1/2}<N^{-3n/2+o(1)}<p^{-1/2-\delta}
$$
for some $\delta=\delta(\varepsilon,n)>0.$ Thus, we can apply Lemma~\ref{lem:B2} with
$$
\gamma_{j}(x)=\frac{T_j(x)}{N^3}.
$$
This implies that
$$
\Bigl|\sum_{\lambda_n=0}^{p-1}T_1(\lambda_1)\ldots T_{n}(\lambda_n)e_p(a\lambda_1\ldots\lambda_n)\Bigr|<N^{3n}p^{-\delta'}.
$$
Inserting this into~\eqref{eqn:W_n 6^n}, we conclude the proof.

\subsection{Proof of Theorem~\ref{thm:Kloost epsilon}}

Let $c\le 1/4$ be the constant that satisfies Theorem~\ref{thm:kI*=kI* small |I|} and take $C=9c^{-2}$. In particular, we can assume that $n>3c^{-1}$.
Clearly, we can also assume that  $N=[p^{9/(cn^2)}]$. For  $k=[cn/3]$ we have
$$
p^{c/k^2}\ge p^{9/(cn^2)}>N.
$$
Thus, for every $j$ the number of solutions of the congruence
$$
y_1^{*}+\ldots+y_k^{*}\equiv y_{k+1}^{*}+\ldots+y_{2k}^{*}\pmod p,\quad y_1,\ldots,y_{2k}\in I_j,
$$
is bounded by $N^{k+o(1)}$. Letting
$$
W_n=\Bigl|\sum_{x_1\in I_1}\ldots\sum_{x_n\in I_n}\alpha_1(x_1)\ldots \alpha_n(x_n)e_p(ax_1^*\ldots x_n^*)\Bigr|,
$$
we have
$$
W_n^{(2k)^n}\le N^{n(2k)^n-2kn}\sum_{x_{11},\ldots,x_{(2k)1}\in I_1}\ldots\sum_{x_{1n},\ldots,x_{(2k)n}\in I_n}e_p(a\{...\})
$$
where $\{...\}$ denotes
$$
\prod_{j=1}^n(x_{1j}^*+\ldots+x_{kj}^*-x_{(k+1)j}^*-\ldots-x_{(2k)j}^*).
$$
We can fix $x_{(k+i)j}$ for all $i=1,\ldots,k$ and $j=1,\ldots,n$ such that for some integers $c_1,\ldots,c_n$ we have
$$
W_n^{(2k)^n}\le N^{n(2k)^n-kn}\,\Bigl|\sum_{x_{11},\ldots,x_{k1}\in I_1}\ldots\sum_{x_{1n},\ldots,x_{kn}\in I_n}e_p(a\{...\})\Bigr|
$$
where $\{...\}$ denotes
$$
(x_{11}^*+\ldots+x_{k1}^*-c_1)\ldots(x_{1n}^*+\ldots+x_{kn}^*-c_n).
$$
Thus,
\begin{equation}
\label{eqn:W_n (2k)^n NNNN}
W_n^{(2k)^n}\le N^{n(2k)^n-kn}\,\Bigl|\sum_{\lambda_1=0}^{p-1}\ldots\sum_{\lambda_n=0}^{p-1}T_1(\lambda_1)\ldots T_n(\lambda_n)e_p(a\lambda_1\ldots \lambda_n)\Bigr|,
\end{equation}
where $T_j(\lambda_j)$ is the number of solutions of the congruence
$$
y_1^{*}+\ldots+y_k^{*}-c_j\equiv\lambda_j\pmod p,\quad (y_1,\ldots,y_k)\in I_j^k.
$$
We have
$$
\sum_{\lambda_j=0}^{p-1}\frac{T_j(\lambda_j)}{N^{k}}\le 1.
$$
Furthermore, by Theorem~\ref{thm:kI*=kI* small |I|}
$$
\sum_{\lambda_j=0}^{p-1}\Bigl(\frac{T_j(\lambda_j)}{N^{k}}\Bigr)^2<\frac{N^{k+o(1)}}{N^{2k}}<N^{-k+o(1)}<p^{-\delta}
$$
and
$$
\prod_{j=1}^n\Bigl(\sum_{\lambda_j=0}^{p-1}\Bigl(\frac{T_j(\lambda_j)}{N^{k}}\Bigr)^2\Bigr)^{1/2}<N^{-kn/2}p^{o(1)}<p^{-1/2-\delta}
$$
for some $\delta=\delta(\varepsilon,n)>0.$ Here we used that
$$
N=[p^{9/(cn^2)}];\qquad N^{kn/2}>N^{cn^2/12}\gg p^{3/4}.
$$
Thus, we can apply Lemma~\ref{lem:B2}, leading to
$$
\sum_{\lambda_1=0}^{p-1}\ldots\sum_{\lambda_n=0}^{p-1}T_1(\lambda_1)\ldots T_n(\lambda_n)e_p(a\lambda_1\ldots \lambda_n)<N^{2kn}p^{-\delta'}
$$
for some $\delta'=\delta'(\varepsilon,n)>0$. Joining this with~\eqref{eqn:W_n (2k)^n NNNN}, we conclude the proof of Theorem~\ref{thm:Kloost epsilon}.

\subsection{Proof of Theorem~\ref{thm:Kloost 1/2}}

Put $N_j=|I_j|$.  Removing, if necessary, intervals $I_j$ with $N_j\le p^{\varepsilon/(2n)}$ we can easily reduce the problem to the case when $N_j>p^{\varepsilon/(2n)}$ for all $j$ and $N_1\ldots N_n\ge p^{1/2+\varepsilon}$. Also note that if $N_j\ge p^{1/2+\varepsilon/(10n)}$, then the claim follows
from Weil's bound for incomplete Kloosterman sums. Thus, we can also assume that $N_j<p^{1/2+\varepsilon/(10n)}$ for all $j$. Then those intervals $I_j$  for which $N_j>p^{1/2}$ we refine
to subintervals of sizes $\approx p^{1/2}$ and thus can assume that $
p^{\varepsilon/(2n)}<N_j<p^{1/2}$ for all $j$. Again we can refine the intervals in an obvious way and eventually reduce the problem to the case when
$$
p^{1/2+\varepsilon}< N_1\ldots N_n< 2p^{1/2+\varepsilon}
$$
and
$$
p^{\varepsilon/(2n)}<N_j<p^{1/2},\quad j=1,2,\ldots, n.
$$

Let
$$
W_n=\Bigr|\sum_{x_1\in I_1}\ldots\sum_{x_n\in I_n}\alpha_1(x_1)\ldots \alpha_n(x_n)e_p(ax_1^*\ldots x_n^*)\Bigl|.
$$
Taking $k=[2/\varepsilon]$ and consequently applying H\"older's inequality $n$ times, we get
\begin{equation}
\label{eqn:W_n (2k)^n}
\begin{split}
W_n^{(2k)^n}\le &(N_1\ldots N_n)^{(2k)^n-2k}\times \\
&\sum_{\lambda_1=0}^{p-1}\ldots\sum_{\lambda_n=0}^{p-1}T_1(\lambda_1)\ldots T_n(\lambda_n)e_p(a\lambda_1\ldots \lambda_n),
\end{split}
\end{equation}
where $T_j(\lambda_j)$ is the number of solutions of the congruence
$$
(y_1^{*}+\ldots+y_k^{*})-(y_{k+1}^{*}+\ldots+y_{2k}^{*})\equiv\lambda_j\pmod p,\quad (y_1,\ldots,y_{2k})\in I_j^{2k}.
$$
Now we observe that
$$
\sum_{\lambda_j=0}^{p-1}\frac{T_j(\lambda_j)}{N_j^{2k}}\le 1.
$$
Furthermore, by Theorem~\ref{thm:kI*}
$$
\sum_{\lambda_j=0}^{p-1}\Bigl(\frac{T_j(\lambda_j)}{N_j^{2k}}\Bigr)^2<\frac{N_j^{4k-2+1/(2k+1)+o(1)}}{N_j^{4k}}<N_j^{-2+1/(2k+1)+o(1)}<p^{-\delta}
$$
and
$$
\prod_{j=1}^n\Bigl(\sum_{\lambda_j=0}^{p-1}\Bigl(\frac{T_j(\lambda_j)}{N_j^{2k}}\Bigr)^2\Bigr)^{1/2}<(N_1\ldots N_n)^{-1}p^{0.5\varepsilon}<p^{-1/2-\delta}
$$
for some $\delta=\delta(\varepsilon,n)>0.$ Thus, we can apply Lemma~\ref{lem:B2}, leading to
$$
\sum_{\lambda_1=0}^{p-1}\ldots\sum_{\lambda_n=0}^{p-1}T_1(\lambda_1)\ldots T_n(\lambda_n)e_p(a\lambda_1\ldots \lambda_n)<(N_1\ldots N_n)^{2k}p^{-\delta'}
$$
for some $\delta'=\delta'(\varepsilon,n)>0$. Joining this with~\eqref{eqn:W_n (2k)^n}, we conclude the proof of Theorem~\ref{thm:Kloost 1/2}.

\section{Proof of Theorem~\ref{thm:Archimed}}

We need the following consequence of~\cite[Theorem 7]{B3}. Let $\mu,\nu$ be positive probability measures on $\R$ supported on $[-1,1]$, $\alpha,\,\beta$ complex functions on $\R$; \, $|\alpha|,|\beta|\le 1$. Let $\xi\in\R,\,|\xi|>1$. Then
\begin{equation}
\label{eqn:B3Measures}
\Bigl|\int\int \alpha(x)\beta(y)e^{ixy\xi}\mu(dx)\nu(dy)\Bigr|\ll |\xi|^{-1/2}\,\|\mu*\varphi_{\delta}\|_{_2}\,\|\nu*\varphi_{\delta}\|_{_2},
\end{equation}
where $\delta=(100|\xi|)^{-1}$ and
$$
\varphi_{\delta}(t)=\left\{\begin{array}{ll}\delta^{-1} \quad {\rm if}\quad t\in[-\frac{\delta}{2},\frac{\delta}{2}],\\ 0\quad {\rm \quad otherwise.}\end{array}\right.
$$ Note that~\eqref{eqn:B3Measures} is very simple, the paper~\cite{B3} contains
also multi-linear versions derived from discretized ring theorem, but we do not need them here.

Let us give a direct proof of~\eqref{eqn:B3Measures}. For the brevity write $\varphi=\varphi_\delta$. Since
$$
\widehat{\varphi}(\lambda)=\frac{2}{\delta\lambda}\,\sin{\frac{\delta\lambda}{2}}
$$
we have, for $|\lambda|<(10\delta)^{-1}=10\xi$ and $|y|\le 1$,
$$
\frac{1}{2}<\widehat{\varphi}(\lambda)\le 1,\qquad\Bigl|\frac{\beta(y)}{\widehat{\varphi}(\xi y)}\Bigr|\le 2.
$$
Write for $x,y\in [-1,1]$
$$
e^{i\xi xy}=\frac{1}{\widehat{\varphi}(\xi y)}\int\varphi(s-x)e^{i\xi sy}ds.
$$
Hence,
\begin{equation}
\label{eqn:propb}
\begin{split}
\Bigl|&\int\int \alpha(x)\beta(y)e^{ixy\xi}\mu(dx)\nu(dy)\Bigr|\le\\
&\int\int\Bigl|\int\frac{\beta(y)}{\widehat{\varphi}(\xi y)}e^{i\xi sy}\nu(dy)\Bigr|\varphi(s-x)\mu(dx)ds=\\
&\quad\int\widehat{\nu_1}(\xi s)(\mu*\varphi)(s)ds,
\end{split}
\end{equation}
where
$$
\frac{d\nu_1}{d\nu}=\frac{\beta(y)}{\widehat{\varphi}(\xi y)},\quad{\rm hence}\quad |\nu_1|\le 2\nu.
$$
Since supp$(\mu*\varphi)\subset [-2,2]$, it follows from~\eqref{eqn:propb} that
$$
\Bigl|\int\int \alpha(x)\beta(y)e^{ixy\xi}\mu(dx)\nu(dy)\Bigr|\le\|\mu*\phi\|_{_2}\,\,\|\widehat{\nu_1}(\xi \cdot)\|_{_{L^2[-2,2]}}.
$$
Next, for $|s|\le 2$
$$
|\widehat{\nu_1}(\xi s)|\le 2|\widehat{\nu_1}(\xi s)\widehat{\varphi}(\xi s)|=2|(\nu_1*\varphi)^{\widehat{}}\,\,(\xi s)|.
$$
Therefore
\begin{equation*}
\begin{split}
\|\widehat{\nu_1}(\xi \cdot)\|_{_L^2[-2,2]} &\le 2|\xi|^{-1/2}\|(\nu_1*\varphi)^{\widehat{}}\,\,\|_{_2}\\
&=2|\xi|^{-1/2}(2\pi)^{1/2}\|\nu_1*\varphi\|_{_2}\\
&\le 4(2\pi)^{1/2}|\xi|^{-1/2}\|\nu*\varphi\|_{_2},
\end{split}
\end{equation*}
and inequality~\eqref{eqn:B3Measures} follows.

Let $e(z)=e^{i z}$. Denoting
$$
S=\sum_{\substack{n_1\sim N_1\\ n_2\sim N_2}}e\Bigl(\frac{1}{n_1}\frac{1}{n_2}\,\xi\Bigr),
$$
we estimate
\begin{equation*}
\begin{split}
|S|^{k_1k_2}&\le N_1^{k_1k_2-k_1}N_2^{k_1k_2-k_2}\times\\
&\times \sum_{\substack{n_{11},\ldots,n_{1k_1}\sim N_1\\n_{21},\ldots,n_{2k_2}\sim N_2}}\quad
\alpha_{\overline{n_1}}\,\beta_{\overline{n_2}}\, e\Bigl(\xi\Bigl(\frac{1}{n_{11}}+\ldots+\frac{1}{n_{1k_1}}\Bigr)\Bigl(\frac{1}{n_{21}}+\ldots+\frac{1}{n_{2k_2}}\Bigr)\Bigr),
\end{split}
\end{equation*}
where $\overline{n_i}=(n_{i1},\ldots,n_{ik_i})$ and
$$
\alpha_{\overline{n_1}}=\alpha_{(\frac{1}{n_{11}}+\ldots+\frac{1}{n_{1k_1}})},\quad \beta_{\overline{n_2}}=\beta_{(\frac{1}{n_{21}}+\ldots+\frac{1}{n_{2k_2}})}
$$
and $|\alpha_{\overline{n_1}}|=|\beta_{\overline{n_2}}|=1.$
Thus, for some complex coefficients
$\alpha(x)$ and $\beta(y)$ with $|\alpha(x)|=|\beta(y)|=1$ we have
$$
|S|^{k_1k_2}\le (N_1N_2)^{k_1k_2}\Bigl|\int\int \alpha(x)\beta(y)e\Bigl(\frac{\xi}{N_1N_2}xy\Bigr)\mu(dx)\nu(dy)\Bigr|,
$$
where $\mu$ is obtained as normalized image measure on $\R$ under the map
$$
\{n_i\sim N_1\}^{k_1}\to \R :\quad (n_{11},\ldots,n_{1k_1})\to N_1\Bigl(\frac{1}{n_{11}}+\ldots+\frac{1}{n_{1k_1}}\Bigr)
$$
and similarly for $\nu$.

Set $\delta=\frac{N_1N_2}{100|\xi|}$. It follows from~\eqref{eqn:B3Measures} that
\begin{equation}
\label{eqn:measureTrigSumpowerk1k2}
|S|^{k_1k_2}\le (N_1N_2)^{k_1k_2}\delta^{1/2}\|\mu*\varphi_{\delta}\|_{_2}\,\|\nu*\varphi_{\delta}\|_{_2}.
\end{equation}
Next we estimate $\|\mu*\varphi_{\delta}\|_{_2}.$ Note that if $I_j$ is a partition of $\R$ in $\delta$-intervals, then
\begin{equation*}
\begin{split}
&\sum_{j}\Bigl|(n_1,\ldots,n_k); \,n_i\sim N\quad {\rm and}\quad \frac{N}{n_1}+\ldots+\frac{N}{n_k}\in I_j\Bigr|\le \\
&\Bigl|(n_1,\ldots,n_{2k}); \,n_i\sim N\quad {\rm and}\quad \Bigl|\frac{1}{n_1}+\ldots+\frac{1}{n_k}-\frac{1}{n_{k+1}}-\ldots-\frac{1}{n_{2k}}\Bigr|<\frac{\delta}{N}\Bigr|.
\end{split}
\end{equation*}
Thus, it follows that
\begin{equation}
\label{eqn:mu convolution Pdelta}
\|\mu*\varphi_{\delta}\|_{_2}^2\sim N_1^{-2k_1}\delta^{-1}T(N_1),
\end{equation}
where
$$
T(N)=\Bigl|(n_1,\ldots,n_{2k}); \,n_i\sim N\,\, {\rm and}\,\, \Bigl|\frac{1}{n_1}+\ldots+\frac{1}{n_k}-\frac{1}{n_{k+1}}-\ldots-\frac{1}{n_{2k}}\Bigr|<\frac{\delta}{N}\Bigr|.
$$
Let us prove that
$$
T(N)<c(k)(\log N)^{4k^2}N^{k}(1+\delta N^{2k-1}).
$$
For $\lambda\in\Q$, denote $J(\lambda)$ the number of solutions of representations of $\lambda$ as
\begin{equation}
\label{eqn:measureRepresentLambda}
\lambda=\frac{1}{n_1}+\ldots+\frac{1}{n_k},\quad n_i\sim N.
\end{equation}
By Lemma~\ref{lem:Karatsuba},
$$
\sum_{\lambda}J(\lambda)^{2}\le c(k)N^k(\log N)^{4k^2}.
$$
Also note that different $\lambda$'s are at least $\sim N^{-2k}$ separated. Hence an interval $I\subset \R$ of size $\delta/N$ contains at most
$1+N^{2k-1}\delta$ elements of the form~\eqref{eqn:measureRepresentLambda}. Therefore,
\begin{equation*}
\begin{split}
T(N)\le&\sum_{n_1,\ldots,n_k\sim N}J\Bigl(\frac{1}{n_1}+\ldots+\frac{1}{n_k}\Bigr)(1+N^{2k-1}\delta)\\
\le &(1+N^{2k-1}\delta)\sum_{\lambda}J(\lambda)^2<c(k)N^k(\log N)^{4k^2}(1+N^{2k-1}\delta),
\end{split}
\end{equation*}
which establishes the required bound for $T(N)$.

Thus, from~\eqref{eqn:mu convolution Pdelta} we get
$$
\|\mu*\varphi_{\delta}\|_{_2}\le c(k_1)(\log N_1)^{2k_1}\delta^{-1/2}N_1^{-k_1/2}(1+\delta N_1^{2k_1-1})^{1/2}.
$$
Similarly,
$$
\|\nu*\varphi_{\delta}\|_{_2}\le c(k_2)(\log N_2)^{2k_2}\delta^{-1/2}N_2^{-k_2/2}(1+\delta N_2^{2k_2-1})^{1/2}.
$$
Inserting these bounds into~\eqref{eqn:measureTrigSumpowerk1k2}, we get
\begin{equation*}
\begin{split}
|S|^{k_1k_2}\le& c(k_1)c(k_2)(\log N_1)^{2k_1}(\log N_2)^{2k_2}(N_1N_2)^{k_1k_2}\times\\
\times &\Bigl(\delta^{-1/2}N_1^{-k_1}+
\delta^{1/2}N_1^{k_1-1}\Bigr)^{1/2}\Bigl(\delta^{-1/2}N_2^{-k_2}+\delta^{1/2}N_2^{k_2-1}\Bigr)^{1/2}.
\end{split}
\end{equation*}
Recalling that $\delta=\frac{N_1N_2}{100|\xi|}$ we conclude the proof.

\section{Some applications}

In this section we apply Theorem~\ref{thm:Archimed} to prove Theorem~\ref{thm:pix-pix-y} on $\pi(x)-\pi(x-y)$ , and apply trilinear exponential sum bounds
to prove Theorem~\ref{thm:linear sqrt log p} on a linear Kloosterman sums and Theorem~\ref{thm:BrunTitch} on Brun-Titchmarsh theorem.

\subsection{Proof of Theorem~\ref{thm:pix-pix-y}}

In view of Huxley's result~\cite{HX} we can assume that $y<x^{7/12+}$. The function
$$
\frac{2(1-\theta)}{12(\theta^{-1}+1)(\theta^{-1}+0.5)+1-\theta}
$$
increases in $\theta\in[0,7/12]$, so we can assume that $y=x^{\theta}$.  Going over the argument of~\cite[p.269]{FrIw1}, (13.56) gives a bound on $R(M,N)$ of the form
$$
\frac{Hy}{MN}\,\Bigl| \sum_{\substack{m\sim M\\ n\sim N}}\alpha_m\beta_n e\Bigl(\frac{hu}{mn}\Bigr)\Bigr|
$$
with $u\sim x$ and $1\le h\le H=MNy^{-1}x^{\varepsilon}$. Here $M,N$ may be chosen arbitrarily with
$$
MN=D>y
$$
(see~\cite[Theorem 12.21]{FrIw1}) and we need to ensure that
$$
R(M,N)<yx^{-\varepsilon}.
$$
We may then state an upper bound
$$
\pi(x)-\pi(x-y)<\frac{2y}{\log D}.
$$
Thus,
\begin{equation}
\label{eqn:measureRMNpix}
\frac{Hy}{MN}\Bigl| \sum_{\substack{m\sim M\\ n\sim N}}\alpha_m\beta_n e\Bigl(\frac{hu}{mn}\Bigr)\Bigr|\le x^{\varepsilon}\Bigl|\sum_{\substack{m\sim M\\ n\sim N}}\alpha_m\beta_n e\Bigl(\frac{\xi}{mn}\Bigr)\Bigr|,
\end{equation}
where
$$
D<x\le \xi<\frac{D}{y}\,x^{1+\varepsilon}.
$$
Take $k$ satisfying
$$
k-\frac{1}{2}<\frac{1}{\theta}<k+\frac{1}{2}
$$
and define $M$ by
$$
\frac{x}{D}=M^{2k-1}.
$$
Let
$$
N=\frac{D}{M}
$$
and choose $l$ such that
$$
N^{2(l-1)}\le\frac{\xi}{D}<N^{2l}.
$$
Hence,
$$
\log N=\log D-\frac{\log\frac{x}{D}}{2k-1}\ge \log y-\frac{\log\frac{x}{y}}{2k-1}=\Bigl(\theta-\frac{1-\theta}{2k-1}\Bigr)\log x>\frac{\theta}{2}\log x
$$
and $l\le \theta^{-1}+1$. Bounding~\eqref{eqn:measureRMNpix} by Theorem~\ref{thm:Archimed} gives
\begin{equation*}
\begin{split}
&x^{\varepsilon}D\Bigl(\frac{\xi}{D}M^{-2k}+\frac{D}{\xi}M^{2(k-1)}\Bigr)^{1/(4kl)}\le\\
& x^{\varepsilon}D\Bigl(\frac{Dx^{\varepsilon}}{y}M^{-1}+M^{-1}\Bigr)^{1/(4kl)}<\\
&x^{2\varepsilon}\Bigl(\frac{D}{y}\Bigr)^{5/4}\Bigl(\frac{x}{D}\Bigr)^{-1/(4k(2k-1)l)}y<\\
&x^{2\varepsilon}\Bigl(\frac{D}{y}\Bigr)^{3/2}x^{-\frac{(1-\theta)\theta}{8(\theta^{-1}+1)(\theta^{-1}+0.5)}}y.
\end{split}
\end{equation*}
Hence we may take
$$
D=y^{1+\frac{1-\theta}{12(\theta^{-1}+1)(\theta^{-1}+0.5)}-\varepsilon'}
$$
implying Theorem~\ref{thm:pix-pix-y}.

\subsection{Proof of Theorem~\ref{thm:linear sqrt log p}}

Denote $\varepsilon=\log N/\log p$. As a consequence of the Weil bound on incomplete Kloosterman sums, we can assume that $\varepsilon<4/7$. Let
$$
\cG=\{x<N:\quad p_1\ge N^{\alpha},\, p_3\ge N^{\beta},\, p_1p_2p_3<N^{1-\beta}\},
$$
where $p_1\ge p_2\ge p_3$ are the largest prime factors of $x$ and
$$
0.1>\alpha> \beta>\frac{1}{\log N}
$$
are parameters to specify. Letting $0.1>\beta_1>\beta$ be another parameter, we observe that
$$
\sum_{\substack{x<N\\ p_1p_2>N^{1-\beta_1}}}1\le \sum_{y<N^{\beta_1}}\sum_{p_1p_2\le N/y}1\le\sum_{y<N^{\beta_1}}\sum_{p_2\le N}\frac{3N}{p_2y\log N}\le 4\beta_1 (\log\log N) N.
$$
Similarly,
$$
\sum_{\substack{x<N\\ p_1p_2p_3 \ge N^{1-\beta}}}1<\sum_{y\le N^{\beta}}\sum_{p_2p_3<N}\frac{4N}{yp_2p_3\log N}<5\beta(\log\log N)^2N.
$$
We also note that the number of positive integers not exceeding $N$ and consisting on products of at most two prime numbers is less than
$$
\frac{2N\log\log N}{\log N}<2\beta N\log\log N.
$$

Hence, we have
\begin{equation*}
\begin{split}
N-|\cG| \le & \frac{2N\log\log N}{\log N}+\\
+&\sum_{\substack{x<N\\ p_1<N^{\alpha}}}1+\sum_{\substack{x<N\\ p_1p_2>N^{1-\beta_1}}}1+ \sum_{\substack{x<N\\ p_1p_2\le N^{1-\beta_1}\\ p_3<N^{\beta}}}1+\sum_{\substack{x<N\\ p_1p_2p_3 \ge N^{1-\beta}}}1
\\&\le \frac{2N\log\log N}{\log N}+ \Psi(N,N^{\alpha})+4\beta_1 N\log\log N\\ &+\sum_{p_1p_2<N^{1-\beta_1}}\Psi\Bigl(\frac{N}{p_1p_2}, N^{\beta}\Bigr)+5\beta N(\log\log N)^2.
\end{split}
\end{equation*}
Here $\Psi(x,y)$, as usual, denotes the number of positive integers $\le x$ having no prime divisors $>y$. By the classical result of de Bruin~\cite{Bruin} if $y>(\log x)^{1+\delta}$, where $\delta>0$ is a fixed constant, then
$$
\Psi(x,y)\le xu^{-u(1+o(1))} \quad {\rm as}\quad u=\frac{\log x}{\log y}\to\infty.
$$
Thus, taking
$$
\alpha=\frac{1}{\log\log p},\qquad  \beta_1=\beta \log\log p,\qquad \frac{2\log\log N}{\log N}<\beta<\frac{1}{(\log\log p)^3},
$$
we have
\begin{equation*}
\begin{split}
N-|\cG| & <\alpha^{\frac{1}{2\alpha}}N+\sum_{p_1p_2<N^{1-\beta_1}}\frac{N}{p_1p_2}\Bigl(\frac{\beta}{\beta_1}\Bigr)^{\frac{\beta_1}{2\beta}}+11\beta N(\log\log p)^2\\
&<\Bigl(\alpha^{\frac{1}{2\alpha}}+(\log\log N)^2\Bigl(\frac{\beta}{\beta_1}\Bigr)^{\frac{\beta_1}{2\beta}}+11\beta (\log\log p)^2\Bigr)N\\
&<12\beta(\log\log p)^2 N.
\end{split}
\end{equation*}
Therefore
$$
\Bigl|\sum_{x<N}e_p(ax^*)\Bigr|\le 12\beta(\log\log p)^2 N+\Bigl|\sum_{x\in \cG}e_p(ax^*)\Bigr|.
$$
We can further assume that
$$
\varepsilon\beta>\frac{1}{\sqrt{\log p}}.
$$
The sum $\sum_{x\in \cG}e_p(ax^*)$ may be bounded by
\begin{equation}
\label{eqn:cuadruple sum}
{\sum_{p_1}}{\sum_{p_2}}{\sum_{p_3}}\Bigl|\sum_{y}e_p(ap_1^*p_2^*p_3^*y^*)\Bigr|,
\end{equation}
where the summations are taken over primes $p_1, p_2, p_3$ and integers $y$ such that
\begin{equation}
\label{eqn:primes p1p2p3}
p_1\ge p_2\ge p_3;\quad p_1\ge N^{\alpha}; \quad p_3\ge N^{\beta};\quad p_1p_2p_3\le N^{1-\beta}
\end{equation}
and
$$
y<\frac{N}{p_1p_2p_3};\quad P(y)\le p_3.
$$
Note that if $t$ and $T$ are such that
\begin{equation}
\label{eqn:t and T}
\Bigl(1-\frac{c}{\log p}\Bigr)p_3<t< p_3,\qquad  \Bigl(1-\frac{c}{\log p}\Bigr)\frac{N}{p_1p_2p_3}<T<\frac{N}{p_1p_2p_3},
\end{equation}
where $c>0$ is any constant, then we can substitute the condition on $y$ with
\begin{equation}
\label{eqn:yyy}
P(y)\le t; \quad y<T
\end{equation}
by changing the sum~\eqref{eqn:cuadruple sum} with an additional term of size at most
$$
\frac{N(\log\log p)^{O(1)}}{\log p}.
$$
Thus, for any $t$ and $T$ satisfying~\eqref{eqn:t and T} we have
$$
\Bigl|\sum_{x\in \cG}e_p(ax^*)\Bigr|<\frac{N(\log\log p)^{O(1)}}{\log p}+
{\sum_{p_1}}{\sum_{p_2}}{\sum_{p_3}}\Bigl|\sum_{y}e_p(ap_1^*p_2^*p_3^*y^*)\Bigr|,
$$
where the summations are taken over primes $p_1, p_2, p_3$ and integers $y$ satisfying~\eqref{eqn:primes p1p2p3} and~\eqref{eqn:yyy}.

Now we split the range
of variations of primes $p_1,p_2,p_3$ into subintervals of the form $[L, L+L(\log p)^{-1}]$ and choosing suitable $t$ and $T$ we obtain that for some numbers $M_1,M_2,M_3$ with
\begin{equation}
\label{eqn:M1M2M3 Kloost restrictions}
M_1>0.5M_2>0.2M_3,\quad M_1>N^{\alpha}, \quad M_3\ge N^{\beta},\quad M_1M_2M_3<N^{1-\beta}
\end{equation}
one has
\begin{equation}
\label{eqn:sum x in G Kloost}
\begin{split}
\Bigl|\sum_{x\in \cG}e_p(ax^*)\Bigr|&<\frac{N(\log\log p)^{O(1)}}{\log p}\\&+
(\log p)^{10}{\sum_{p_1\in I_1}}{\sum_{p_2\in I_2}}{\sum_{p_3\in I_3}}\Bigl|\sum_{\substack{y\le M\\ P(y)\le M_3}}e_p(ap_1^*p_2^*p_3^*y^*)\Bigr|,
\end{split}
\end{equation}
where
$$
I_j=\Bigl[M_j, M_j+\frac{M_j}{\log p}\Bigr], \qquad M=\frac{N}{M_1M_2M_3}\ge N^{\beta}.
$$
Denote
$$
W=\sum_{p_1\in I_1}\sum_{p_2\in I_2}\sum_{p_3\in I_3}\Bigl|\sum_{\substack{y\le M\\ P(y)\le M_3}}e_p(ap_1^*p_2^*p_3^*y^*)\Bigr|.
$$
Applying the Cauchy-Schwarz inequality, we get
$$
W^2\le M_1M_2M_3\sum_{y\le M}\sum_{z\le M}\Bigl|\sum_{p_1\in I_1}\sum_{p_2\in I_2}\sum_{p_3\in I_3}e_p\Bigl(ap_1^*p_2^*p_3^*(y^*-z^*)\Bigr)\Bigr|.
$$
Taking into account the contribution from $y=z$ and then fixing $y\not=z$ we get, for some $b\not\equiv 0\pmod p$,
\begin{equation}
\label{eqn:W^2 Kloost}
W^2\le \frac{N^2}{M}+NM|S|,
\end{equation}
where
$$
|S|=\Bigl|\sum_{p_1\in I_1}\sum_{p_2\in I_2}\sum_{p_3\in I_3}e_p(bp_1^*p_2^*p_3^*)\Bigr|.
$$
Define integers $k_1,k_2,k_3$ such that
$$
p^{\frac{1}{2k_j}}\le M_j<p^{\frac{1}{2(k_j-1)}},\quad j=1,2,3.
$$
From~\eqref{eqn:M1M2M3 Kloost restrictions} and the choice of $\alpha$ and $\beta$, it follows, in particular, that
\begin{equation}
\label{eqn:kj Kloost}
k_1<\frac{1}{\varepsilon \alpha}<(\log p)^{1/2},\quad k_2,k_3<\frac{1}{\varepsilon\beta}<(\log p)^{1/2}.
\end{equation}
We further take even integers $l_j\in \{k_j,k_j+1\}$, \, $(j=1,2,3)$ and define
$$
\eta_j(\lambda)=\Bigl|\{(x_1,\ldots,x_{l_j})\in (I_j\cap\cP)^{l_j}:\, x_1^*-x_2^*+\ldots-x_{l_j}^*\equiv \lambda\pmod p\}\Bigr|.
$$
We have
\begin{equation}
\label{eqn:sum of eta j}
\sum_{\lambda=0}^{p-1}\eta_j(\lambda)<M_j^{l_j}
\end{equation}
and, applying Theorem~\ref{thm:kI*=kI* I=[1,N] prime},
\begin{equation}
\label{eqn:sum eta square Kloost}
\begin{split}
\sum_{\lambda=0}^{p-1}\eta_j(\lambda)^2<&M_j^{2(l_j-k_j)}M_j^{k_j}(2k_j)^{k_j}\Bigl(\frac{M_j^{2k_j-1}}{p}+1\Bigr)\\
<& (2k_j)^{k_j}M_j^{2l_j}p^{-1/2}.
\end{split}
\end{equation}
Next we apply consequently  the H\"{o}lder inequality. We  get
$$
|S|^{l_1}< (M_2M_3)^{l_1-1}\sum_{p_2\in I_2}\sum_{p_3\in I_3}\Bigl|\sum_{p_1\in I_1}e_p(bp_1^*p_2^*p_3^*)\Bigr|^{\l_1}
$$
and since we took  $l_j$ even, we further get
\begin{equation*}
\begin{split}
|S|^{l_1}<&(M_2M_3)^{l_1-1}\sum_{\lambda_1=0}^{p-1}\eta_1(\lambda_1)\Bigl[\sum_{p_2\in I_2}\sum_{p_3\in I_3}e_p(b\lambda_1p_2^*p_3^*)\Bigr]\\
<&(M_2M_3)^{l_1-1}\sum_{\lambda_1=0}^{p-1}\eta_1(\lambda_1)\sum_{p_3\in I_3}\Bigl|\sum_{p_2\in I_2}e_p(b\lambda_1p_2^*p_3^*)\Bigr|.
\end{split}
\end{equation*}
Applying the H\"{o}lder inequality and using~\eqref{eqn:sum of eta j}, we obtain
\begin{equation*}
\begin{split}
|S|^{l_1l_2}<&(M_2M_3)^{(l_1-1)l_2}M_1^{(l_2-1)l_1}M_3^{l_2-1}\sum_{\lambda_1=0}^{p-1}\eta_1(\lambda_1)\sum_{p_3\in I_3}\Bigl|\sum_{p_2\in I_2}e_p(b\lambda_1p_2^*p_3^*)\Bigr|^{l_2}\\
=&M_1^{l_1(l_2-1)}M_2^{l_2(l_1-1)}M_3^{l_1l_2-1}\sum_{\lambda_1=0}^{p-1}\sum_{\lambda_2=0}^{p-1}\eta_1(\lambda_1)\eta_2(\lambda_2)\Bigl|\sum_{p_3\in I_3}e_p(b\lambda_1\lambda_2p_3^*)\Bigr|.
\end{split}
\end{equation*}
We again apply the H\"{o}lder inequality and use~\eqref{eqn:sum of eta j},
\begin{equation*}
\begin{split}
|S|^{l_1l_2l_3}&<M_1^{l_1l_3(l_2-1)}M_2^{l_2l_3(l_1-1)}M_3^{l_3(l_1l_2-1)}(M_1^{l_1}M_2^{l_2})^{l_3-1}\times\\
&\times\Bigl|\sum_{\lambda_1=0}^{p-1}\sum_{\lambda_2=0}^{p-1}\sum_{\lambda_3=0}^{p-1}
\eta_1(\lambda_1)\eta_2(\lambda_2)\eta_3(\lambda_3)e_p(b\lambda_1\lambda_2\lambda_3)\Bigr|\\
&=M_1^{l_1l_2l_3-l_1}M_2^{l_1l_2l_3-l_2}M_3^{l_1l_2l_3-l_3}|S_1|,
\end{split}
\end{equation*}
where
$$
|S_1|=\Bigl|\sum_{\lambda_1=0}^{p-1}\sum_{\lambda_2=0}^{p-1}\sum_{\lambda_3=0}^{p-1}
\eta_1(\lambda_1)\eta_2(\lambda_2)\eta_3(\lambda_3)e_p(b\lambda_1\lambda_2\lambda_3)\Bigr|.
$$
We apply Lemma~\ref{lem:B2} with $n=3$ and
$$
\gamma_j(\lambda)=\frac{\eta_j(\lambda)}{M^{\l_j}}.
$$
From~\eqref{eqn:sum of eta j} it follows that
$$
\|\gamma_i\|_1=\sum_{\lambda=0}^{p-1}\frac{\eta_j(\lambda)}{M^{\l_j}}\le 1.
$$
From~\eqref{eqn:sum eta square Kloost} and~\eqref{eqn:kj Kloost} it follows that
$$
\|\gamma_i\|_2=\Bigl(\sum_{\lambda=0}^{p-1}\Bigl(\frac{\eta_j(\lambda)}{M^{\l_j}}\Bigr)^2\Bigr)^{1/2}<p^{-1/5},
$$
and, in particular,
$$
\prod_{i=1}^{n}\|\gamma_i\|_2<p^{-1/2-1/10}.
$$
Thus, Lemma~\ref{lem:B2} applies and leads to
$$
|S_1|<M_1M_2M_3 p^{-c},
$$
for some absolute constant $c>0$. Consequently
$$
|S|<M_1M_2M_3p^{-c_1/(k_1k_2k_3)}<M_1M_2M_3 m^{-c_2\varepsilon^3\alpha\beta^2}.
$$
Inserting this into~\eqref{eqn:W^2 Kloost} and taking into account that
$$
M>N^{\beta}=p^{\varepsilon\beta}>\exp(\sqrt{\log p}),
$$
we get
$$
W<\frac{N}{\exp(0.5\sqrt{\log p})}+N m^{-c_3\varepsilon^3\alpha\beta^2}.
$$
Inserting this into~\eqref{eqn:sum x in G Kloost}, we get
$$
\Bigl|\sum_{x\in \cG}e_p(ax^*)\Bigr|<\frac{N(\log\log p)^{O(1)}}{\log p}+(\log p)^{10}N m^{-c_3\varepsilon^3\alpha\beta^2}.
$$
Therefore
$$
\Bigl|\sum_{x<N}e_p(ax^*)\Bigr|\le \frac{N(\log\log p)^{O(1)}}{\log p}+12\beta(\log\log p)^2 N+(\log p)^{10}N m^{-c_3\varepsilon^3\alpha\beta^2}.
$$
Thus, taking
$$
\beta=\frac{C\log\log p}{\varepsilon^{3/2}(\log p)^{1/2}},
$$
with sufficiently large constant $C$, we obtain
$$
\Bigl|\sum_{x<N}e_p(ax^*)\Bigr|\ll \frac{(\log\log p)^3}{(\log p)^{1/2}}\,\varepsilon^{-3/2}N,
$$
which finishes the proof.

\subsection{Proof of Theorem~\ref{thm:BrunTitch}}

We only treat the case when $q$ is prime, but the argument generalizes.

We follow~\cite{FrIw}. Denote
\begin{equation*}
\begin{split}
&\cA=\{n\le x; \, n\equiv a\pmod q\},\\
&\cA_d=\{n\in\cA; \, n\equiv 0\pmod d\},\\
&S(\cA,z)=|\{n\in \cA;\, (n,p)=1\,\, {\rm for}\,\, p<z, (p,q)=1\}|,\\
& D=\frac{x^{1-\varepsilon}}{q}.
\end{split}
\end{equation*}

Let $D^{1/5}<w<y<z=(x/q)^{1/3}$, where $w$ and $y$ to specify, and write using Buchstab's identity
\begin{equation}
\label{eqn:Buchstab}
\begin{split}
S(\cA,z)=&S(\cA,w)-\sum_{y\le p<z}S(\cA_p,z)-\sum_{w\le p<y}S(\cA_p,w)\\
+&\sum_{w\le p_1<p_2<y}S(\cA_{p_1p_2},p_2).
\end{split}
\end{equation}
Applying the basic estimates of the linear sieve on each of the terms in~\eqref{eqn:Buchstab}
leads to the bound
\begin{equation}
\label{eqn:BT c=2 S(A,z)}
S(\cA,z)<\frac{2x}{\phi(q)\log D},
\end{equation}
see the discussion in~\cite{FrIw} and also~\cite[p.265]{FrIw1}. In particular this involves bounding
\begin{equation}
\label{eqn:BTc=2}
S(\cA_{p_1p_2}, p_2)<\frac{2x}{\phi(q)p_1p_2\log D_{12}}\quad {\rm with} \quad D_{12}=\frac{D}{p_1p_2}.
\end{equation}
Here $D_{12}$ is the level of distribution for the sequence $\cA_{p_1p_2}$. The idea from~\cite{FrIw}
is to improve on~\eqref{eqn:BTc=2}, in the average over $p_1,p_2$, by increasing the level $D_{12}$ to some level
$D_{12}'.$

More precisely, define the reminders
\begin{equation}
\label{eqn:sieveError}
R_{p_1,p_2,d}=|\cA_{dp_1p_2}|-\frac{x}{qdp_1p_2}
\end{equation}
that appear as the error terms in the sieving process. The strategy is to bound the collected contribution of $R_{p_1,p_2,d}$,
performing the summation over $p_1,p_2$.

Subdivide $[w,y]$ in dyadic ranges and estimate
$$
\sum_{\substack{p_1\sim P_1 \\p_2\sim P_2}} S(\cA_{p_1p_2}, p_2)
$$
for fixed $P_1,P_2$. We introduce $D_{12}'=D_{12}'(P_1,P_2)>D_{12}$ such that
\begin{equation}
\label{eqn:sievesumerror}
\sum_{d<D_{12}'}\Bigl|\sum_{\substack{p_1\sim P_1 \\p_2\sim P_2}} R_{p_1,p_2,d}\Bigr|<\frac{x^{1-\varepsilon}}{q}\,P_1P_2.
\end{equation}
With $D_{12}'$ as sieving limit,~\eqref{eqn:BTc=2} improves to
$$
S(\cA_{p_1p_2}, p_2)<\frac{2x}{\phi(q)p_1p_2\log D_{12}'}
$$
on average over $p_1\sim P_1,\, p_2\sim P_2$, provided $p_2^{3}>D_{12}'$. The gain in~\eqref{eqn:BT c=2 S(A,z)} becomes then of the order
\begin{equation}
\label{eqn:sieve qain 1}
\begin{split}
&\frac{x}{\phi(q)}\sum_{w\le p_1<p_2<y}\frac{1}{p_1p_2}\Bigl(\frac{1}{\log D_{12}}-\frac{1}{\log D_{12}'(p_1,p_2)}\Bigr)\\
& \le\frac{x}{\phi(q)}\sum_{w\le p_1<p_2<y}\frac{1}{p_1p_2}\,\frac{\log\frac{D_{12}'(p_1,p_2)}{D_{12}}}{(\log D_{12})^2}.
\end{split}
\end{equation}
Using the analysis from~\cite{FrIw} in order to express~\eqref{eqn:sieveError} as exponential sums, we obtain following bound on the
left hand side of~\eqref{eqn:sievesumerror}
\begin{equation}
\label{eqn:sieve to exp sums}
\sum_{d<D_{12}'}\sum_{0<|h|<H}\Bigl| \sum_{\substack{p_1\sim P_1\\ p_2\sim P_2}}(qdp_1p_2)^{-1}\widehat{f}\Bigl(\frac{h}{qdp_1p_2}\Bigr)e_{q}(-ahd^*p_1^*p_2^*)\Bigr|
\end{equation}
up to admissible error term. Here
$$
H=qdP_1P_2x^{2\varepsilon-1}
$$
and $f\ge 0$ is supported on $x^{1-\varepsilon}<t<x+x^{1-\varepsilon}$ satisfying $\widehat{f}(0)=x$ and $t^jf^{(j)}(t)\ll x^{\varepsilon}$ for $j\ge 0.$
Standard manipulations permit to express~\eqref{eqn:sieve to exp sums} in terms of trilinear sums ($D_{12}\le \widetilde{D}<D_{12}'$)
$$
\frac{x}{q\widetilde{D}P_1P_2}H\Bigl|\sum_{\substack{d\in I_0\\ p_1\in I_1\\ p_2\in I_2}}\alpha_d\beta_{p_1}\gamma_{p_2}e_q(-ahd^*p_1^*p_2^*)\Bigr|
$$
with $|\alpha_d|,|\beta_{p_1}|,|\gamma_{p_2}|\le 1$ and
$$
I_0\subset [\widetilde{D}, 2\widetilde{D}],\quad I_1\subset [P_1,2P_1],\quad I_2\subset [P_2, 2P_2]
$$
intervals. Note that these intervals always be enlarged to
$$
I_0=[\widetilde{D}, 2\widetilde{D}],\quad I_1=[P_1,2P_1],\quad I_2=[P_2, 2P_2].
$$
Take $\delta=(1-\theta)/5$ and $w=x^{\delta},\, y=x^{2\delta}$. Note that
$$
\widetilde{D}>D_{12}>\frac{D}{y^2}>x^{0.9\delta}.
$$
Performing the Karatsuba amplification followed by a trilinear estimate gives a saving of a factor $x^{c'\delta^3}$, hence
$$
\frac{x}{q\widetilde{D}P_1P_2}H\Bigl|\sum_{\substack{d\in I_0\\ p_1\in I_1\\ p_2\in I_2}}\alpha_d\beta_{p_1}\gamma_{p_2}e_q(-ahd^*p_1^*p_2^*)\Bigr|<
\frac{x^{1-c'\delta^3}}{q}H<x^{2\varepsilon-c'\delta^3}D_{12}'P_1P_2
$$
which is less than $x^{1-\varepsilon} q^{-1}$ for
$$
D_{12}'=\frac{x^{1+c''\delta^3}}{qP_1P_2}=D_{12} x^{c''\delta^3}.
$$
The condition $w^{5}>D_{12}'$ is satisfied. Thus, returning to~\eqref{eqn:sieve qain 1}, we obtain the saving
\begin{equation*}
\begin{split}
\frac{x}{\phi(q)}\sum_{x^{\delta}<p_1<p_2<x^{2\delta}}\frac{1}{p_1p_2}\,\frac{\log (x^{c''\delta^3})}{\delta^2(\log x)^2}\gg\frac{x\delta}{\phi(q)\log x}\sim \frac{x\delta^2}{\phi(q)\log\frac{x}{q}}
\end{split}
\end{equation*}
and since $\varepsilon$ is arbitrarily small, the result follows.

\section{Comments}

There is an alternative approach to Theorem~\ref{thm:BrunTitch} following the proof of~\cite[Theorem 13.1]{FrIw1}. On~\cite[p.262]{FrIw1}, there is a bound
$$
R(M,N)\ll x^{\varepsilon}\Bigl\{(M,N)-{\rm bilinear \,\,Kloosterman \,\, sum}\Bigr\}+\frac{x^{1-\varepsilon}}{q}.
$$
According to~\cite[Theorem 12.21]{FrIw1}, we may consider any factorization $D=MN$.
Let $x/q=x^{\delta}$ and take $k$ such that
$$
\frac{1}{2k-1}\le\frac{\delta}{2}<\frac{1}{2k-3}.
$$
Let
$$
N=q^{1/(2k-1)},\quad M=\frac{D}{N}.
$$
From our Theorem~\ref{thm:Kloost Karatsuba range} it follows that
$$
\Bigl\{(M,N)-{\rm bilinear \,\,Kloosterman \,\, sum}\Bigr\}<MN^{1-c/k^2}<D^{1-c\delta^2}.
$$
From condition $D^{1-c\delta^2}<x^{1-\varepsilon}q^{-1}$ we obtain,
$$
\pi(x;q,a)<\frac{(2-c\delta^2)x}{\phi(q)\log(x/q)}.
$$

As we have mentioned in the introduction one can prove that if $I\subset \F_p^*$ with $|I|<p^{1/2}$, then
$$
|I^{-1}+I^{-1}+I^{-1}|>|I|^{1.55+o(1)}.
$$
This is better than what one gets from Corollary~\ref{cor:thm kI*} for $k=3$. Let us prove this bound.
We can assume that $|I|=N>p^{1/3}$, as Corollary~\ref{cor:CillGar}
implies a better bound
$$
|I^{-1}+I^{-1}|>|I|^{2+o(1)}.
$$
Consider the congruence
$$
\frac{1}{x_1}+\frac{1}{x_2}+\frac{1}{x_3}\equiv \frac{1}{x_4}+\frac{1}{x_5}+\frac{1}{x_6}\pmod p,\quad x_1,\ldots,x_6\in I.
$$
The number $J_6$ of this congruence can be bounded by
$$
J_6\le (J_4J_8)^{1/2}.
$$
From Theorem~\eqref{thm:kI*} it follows that $J_8<N^{32/5+o(1)}$, and from Corollary~\ref{cor:CillGar} we have that
$$
J_4<\frac{N^{7/2+o(1)}}{p^{1/2}}.
$$
Thus,
$$
J_6<\frac{N^{99/20+o(1)}}{p^{-1/4}}.
$$
Using the relationship between the number of solutions of a congruence and the cardinality of the corresponding set, we conclude that
$$
|I^{-1}+I^{-1}+I^{-1}|\ge N^{21/20+o(1)}p^{1/4}>N^{1.55+o(1)}.
$$

We remark that the arguments in the proof of Theorem~\ref{thm:kI*} also give the following.

\begin{proposition}
\label{prop:th1general}
For any fixed positive integer constants $r$ and $k$ the number $J_{2k}^{(r)}$ of solutions of the congruence
$$
\frac{1}{x_1^r}+\ldots+\frac{1}{x_k^r}\equiv \frac{1}{x_{k+1}^r}+\ldots+\frac{1}{x_{2k}^r}\pmod p,\qquad x_1,\ldots, x_{2k}\in I,
$$
satisfies the bound
$$
J_{2k}^{(r)}<\Bigl(|I|^{2k^2/(k+1)}+\frac{|I|^{2k}}{p}\Bigr)|I|^{o(1)}.
$$
\end{proposition}

Thus, in particular we obtain
\begin{corollary}
\label{cor:th1general}
Let $r_1,r_2,k_1,k_2$ be fixed positive integer constants, $I_1=[a_1+1,a_1+N_1]$,\, $I_2=[a_2+1, a_2+N_2]$ and
$$
N_1<p^{\frac{k_1+1}{2k_1}},\quad N_2<p^{\frac{k_2+1}{2k_2}}.
$$
Then for any complex coefficients $\alpha_1(x_1), \alpha_2(x_2)$ with $|\alpha_i(x_i)|\le 1$ one has
\begin{equation*}
\begin{split}
\max_{(a,p)=1}\Bigl|\sum_{x_1\in I_1}\sum_{x_2\in I_2}&\alpha_1(x_1)\alpha_2(x_2)e_p(ax_1^{-r_1}x_2^{-r_2})\Bigr|<\\
&\Bigl(p^{\frac{1}{2k_1k_2}}N_1^{-\frac{1}{k_2(k_1+1)}}N_2^{-\frac{1}{k_1(k_2+1)}}\Bigr)(N_1N_2)^{1+o(1)}.
\end{split}
\end{equation*}
\end{corollary}

One may apply the bilinear sums of Corollary~\ref{cor:th1general} together with Vaughan's formula~\cite{Vau} to bound the corresponding sums over primes, similar to those in~\cite[Theorems A1, A9]{B1} and~\cite[Corollary 1.5]{FM}.

\begin{corollary}
\label{cor:th1KloostPrimespower}
Let $r\in\Z_{+}$ and  $p>N>p^{1/2+\varepsilon}$ for some $\varepsilon>0$. Then
$$
\max_{(a,p)=1}\Bigl|\sum_{\substack{x<N\\ x\,\, {\rm prime}}}e_p(ax^{-r})\Bigr|<N^{1-\delta}
$$
for some $\delta=\delta(\varepsilon;r)>0.$
\end{corollary}

Next we remark that the result from~\cite{Gar} implies the bound
$$
\max_{(a,p)=1}\Bigl|\sum_{\substack{x<p\\ x\,\, {\rm prime}}}e_p(ax^{-1})\Bigr|<p^{15/16+o(1)}.
$$
Our Corollary~\ref{cor:th1general} leads to
\begin{corollary}
\label{cor:th1KloostPrimespowerFull range}
For any fixed positive integer constant $r$ the following bound holds:
$$
\max_{(a,p)=1}\Bigl|\sum_{\substack{x<p\\ x\,\, {\rm prime}}}e_p(ax^{-r})\Bigr|<p^{23/24+o(1)}.
$$
\end{corollary}

Let us prove it. It suffices to establish the bound
$$
\Bigl|\sum_{n\le p}\Lambda(n)e_p(an^{-r})\Bigr| <p^{23/24+o(1)}
$$
and then the result follows by partial summation. Here $\Lambda(n)$ is the Mangoldt function.

Below we use $A\lessapprox B$ to mean that $A<Bp^{o(1)}$. From the Vaughan's identity (see~\cite[Chapter 24]{Dav}), we have
$$
\sum_{n\le p}\Lambda(n)e_p(an^{-r})\lessapprox W_1+W_2+W_3+W_4,
$$
where
\begin{eqnarray*}
&& W_1=\Bigl|\sum_{n\le U}\Lambda(n)e_p(an^{-r})\Bigr|;\\
&& W_2=\sum_{n\le UV}\Bigl|\sum_{m\le p/n}e_p(an^{-r}m^{-r})\Bigr|;\\
&& W_3=\sum_{n\le V}\Bigl|\sum_{m\le p/n}(\log m) e_p(an^{-r}m^{-r})\Bigr|;\\
&& W_4=\sum_{U<n\le p/V}\Bigl|\sum_{V<m\le
p/n}\beta_me_p(an^{-r}m^{-r})\Bigr|.
\end{eqnarray*}
Here $U\ge 2, V\ge 2$ are parameters with $UV\le p,$
$$
\beta_m=\sum_{\substack{d|m\\ d\le V}}\mu(d).
$$
Below we shall also use Weil's bound of the form
$$
\Bigl|\sum_{x=1}^{p-1}e_p(ax^{-r}+bx)\Bigr|\ll p^{1/2}.
$$
More precisely, we shall use a consequence of this bound, namely if $I$ is an interval in $\F_p$ then
$$
\Bigl|\sum_{x\in I}e_p(ax^{-r})\Bigr|\lessapprox p^{1/2}.
$$

We take $U=V=p^{1/3}$ and estimate $W_1$ trivially:
$$
W_1\lessapprox U= p^{1/3}.
$$

To estimate $W_2$  we split the range of summation over $n$ into dyadic
intervals and get, for some $L\le p^{2/3},$
$$
W_2\lessapprox \sum_{L\le n\le 2L}\Bigl|\sum_{m\le
p/n}e_p(an^{-r}m^{-r})\Bigr|.
$$
If $L < p^{1/3}$ then we apply Weil's bound to the sum over $m$ and get
$$
W_2\lessapprox Lp^{1/2}\lessapprox p^{5/6}.
$$
If $p^{1/3}<L < p^{2/3}$ then using the standard smoothing argument we
extend the summation over $m$ to $m\le p/L$ and apply Corollary~\ref{cor:th1general} with $k_1=k_2=2$ to get
$$
W_2\lessapprox p^{23/24}.
$$

To estimate $W_3$ we use partial summation (to the sum over $m$) and Weil's bound and get
$$
W_3\lessapprox\sum_{n< p^{1/3}}p^{1/2}\lessapprox p^{5/6}.
$$

To estimate $W_4,$ we split the range of summation over $n$ into
dyadic intervals and get, for some $p^{1/3}<L\le p^{2/3},$
$$
W_4\lessapprox \sum_{L\le n< 2L}\Bigl|\sum_{p^{1/3}<m\le
p/n}\beta_me_p(an^{-r}m^{-r})\Bigr|.
$$
Applying the smoothing argument to extend the sum over $m$ to $m<p/L$ and using Corollary~\ref{cor:th1general} with $k_1=k_2=2$, we get
$$
W_4\lessapprox p^{23/24}
$$
and Corollary~\ref{cor:th1KloostPrimespowerFull range} follows.


\begin{thebibliography}{100}
\bibitem{ACZ} A. Ayyad, T. Cochrane and Z. Zheng,
`The congruence $x_1x_2 \equiv x_3x_4 \pmod p$, the equation
$x_1x_2 = x_3x_4$ and the mean value of character sums',
{\it J. Number Theory\/}, {\bf 59} (1996), 398--413.

\bibitem{Bak} R. C. Baker, `Kloosterman sums with prime variable', {\it Acta Arith.\/}, (to appear).

\bibitem{BHW} U. Betke, M. Henk and J. M. Wills,
`Successive-minima-type inequalities', {\it Discr. Comput. Geom.\/},
{\bf 9} (1993), 165--175.

\bibitem{BorShaf} Z. I. Borevich and I. R. Shafarevich, {\it Number Theory\/}, Academic Press, 1966.

\bibitem{B1} J. Bourgain, `More on sum-product phenomenon in prime fields and its applications',
{\it Int. J. Number Theory}  {\bf 1}  (2005),  1--32.

 \bibitem{B2} J. Bourgain, `Multilinear exponential sums in prime fields under optimal entropy condition on the sources',
{\it Geom. Funct. Anal.} {\bf 18} (2009), 1477--1502.

\bibitem{B3} J. Bourgain, `The discretized sum-product and projection theorems',
{\it J. Anal. Math.}  {\bf 112}  (2010),  193--236.

\bibitem{BGKS1} J.~Bourgain, M.~Z.~ Garaev, S. V. Konyagin and
I. E. Shparlinski,
`On the hidden shifted power problem',
{\it SIAM J. Comp.}, (to appear).


\bibitem{BGKS2} J.~Bourgain, M.~Z.~Garaev, S. V. Konyagin and
I. E. Shparlinski,
`On congruences with products of variables from short intervals and  applications',
{\it Proc. Steklov Inst. Math.} (to appear).


\bibitem{Bruin} N. G. de Bruijn, `On the number of positive integers $\le x$ and free prime factors $>y$, II',
{\it Indag. Math.} {\bf 28} (1966), 239–-247.

\bibitem{Burgess} D. A. Burgess, `Partial Gaussian sums', {\it  Bull. London Math. Soc.\/},  {\bf 20} (1988), 589--592.


\bibitem{CillGar} J. Cilleruelo and
M. Z. Garaev, `Concentration of points on two and three dimensional
modular hyperbolas and applications', {\it  Geom. Func.
Anal.\/},  {\bf 21} (2011), 892--904.

\bibitem{Dav} H. Davenport, ``Multiplicative Number Theory". Third edition, revised by H.~L.~Montgomery. Springer-Verlag, New York, 2000.

\bibitem{FM} E. Fouvry and Ph. Michel `Sur certaines sommes d'exponentielles sur les nombres premiers', {\it
Ann. Sci. École Norm. Sup.\/} {\bf 31} (1998), 93--130.

\bibitem{FouSh} E. Fouvry and I. E. Shparlinski `On a ternary quadratic form over primes', {\it
Acta Arith.\/} {\bf 150} (2011), 285--314.


\bibitem{FrIw0} J. Friedlander and H. Iwaniec,
`Estimates for character sums',
{\it Proc. Amer. Math. Soc.\/}, {\bf 119} (1993), 365--372.


\bibitem{FrIw} J. Friedlander and H. Iwaniec, {\it The Brun-Titchmarsh theorem},
{\it Analytic number theory} (Kyoto, 1996), 85–-93, London Math. Soc. Lecture Note Ser., 247, Cambridge Univ. Press,
Cambridge, 1997.

\bibitem{FrIw1} J. Friedlander and H. Iwaniec, {\it Opera de Cribro},
American Mathematical Society, Colloquium Publications, 57, 2010.


\bibitem{Gar} M. Z. Garaev, `Estimation of Kloosterman sums with primes and its application', {\it
Math. Notes\/} {\bf 88}:3 (2010), 365--373.

\bibitem{HB}  D. R. Heath-Brown, `Almost-primes in arithmetic progressions and short intervals', {\it
Math. Proc. Cambridge Philos. Soc.} {\bf 83}  (1978), no. 3,
357--375.

\bibitem{HX} M. N. Huxley,
`On the difference between consecutive primes', {\it
Invent. Math.} {\bf 15} (1972), 164-–170.

\bibitem{IRRO} K. Ireland and M. Rosen, {\it A classical introduction to modern number theory\/},
Springer, 1990.

\bibitem{KarKKR} A. A. Karatsuba, `New estimates of short Kloosterman sums',  {\it Math. Notes} {\bf 88}:3 (2010), 347--359.
Prepared by E. A. Karatsuba, M. A. Korolev and I. S. Rezvyakova
on notes and drafts of A. A. Karatsuba.

\bibitem{Kar1} A. A. Karatsuba, `Analogues of Kloosterman sums',  {\it Izv. Math.} {\bf 59}:5 (1995), 971--981.

\bibitem{Kar2} A. A. Karatsuba, `Fractional parts of functions of a special form',  {\it Izv. Math.} {\bf 59}:4 (1995), 721--740.

\bibitem{Kor} M. A. Korolev, `Incomplete Kloosterman sums and their applications',  {\it Izv. Math.} {\bf 64}:6 (2000), 1129–-1152.

\bibitem{Kor1} M. A. Korolev, `Short Kloosterman sums with weights',  {\it Math. Notes} {\bf 88}:3 (2010), 374–-385.


\bibitem{Luo} W. Luo, `Bounds on incomplete multiple Kloosterman sums',
{\it J. Number Theory\/}, {\bf 75} (1999), 41--46.

\bibitem{Pras} V. V. Prasolov, {\it Polynomials}, Algorithms and Computation in Mathematics, {\bf 11}, Springer-Verlag, Berlin, 2004.

\bibitem{Shp1} I. E. Shparlinski,
`Bounds on incomplete multiple Kloosterman sums',
{\it J. Number Theory\/}, {\bf 126} (2007), 68--73.

\bibitem{TaoVu}
T. Tao and V. Vu, {\it Additive combinatorics\/}, Cambridge Stud.
Adv. Math., {\bf 105}, Cambridge University Press, Cambridge, 2006.

\bibitem{Vau} R. C. Vaughan, `Sommes trigonom\'etriques sur les nombres premiers', {\it C.R. Acad. Sci. Paris} Ser. A {\bf 285}, 981–-983.

\end{thebibliography}
\end{document}